\newif\ifjems

\newif\ifams
\amstrue %

\ifjems
\documentclass[12pt, a4paper, twoside]{article}
\usepackage{amsmath,amsthm,amssymb}
\usepackage{times}
\usepackage{enumerate}

\def\titlerunning#1{\gdef\titrun{#1}}
\makeatletter
\def\author#1{\gdef\autrun{\def\and{\unskip, }#1}\gdef\@author{#1}}
\def\address#1{{\def\and{\\\hspace*{18pt}}\renewcommand{\thefootnote}{}%
\footnote {#1}}%
\markboth{\autrun}{\titrun}}
\makeatother
\def\email#1{e-mail: #1}
\def\subjclass#1{{\renewcommand{\thefootnote}{}%
\footnote{\emph{Mathematics Subject Classification (2010):} #1}}}
\def\keywords#1{\par\medskip
\noindent\textbf{Keywords.} #1}

\newtheorem{theorem}{Theorem}[section]
\newtheorem{corollary}[theorem]{Corollary}
\newtheorem{lemma}[theorem]{Lemma}
\newtheorem{proposition}[theorem]{Proposition}

\theoremstyle{definition}
\newtheorem{definition}[theorem]{Definition}
\newtheorem{remark}[theorem]{Remark}

\numberwithin{equation}{section}

\fi

\ifams
\documentclass[11pt, a4paper]{amsart}

\title[Non-renormalized solutions to the continuity equation]{Non-renormalized solutions \\ to the continuity equation}
\date{\today}

\author{Stefano Modena}
\address{Institut f\"ur Mathematik, Universit\"at Leipzig, D-04109 Leipzig, Germany}
\email{stefano.modena@math.uni-leipzig.de}

\author{L\'aszl\'o Sz\'ekelyhidi Jr.}
\address{Institut f\"ur Mathematik, Universit\"at Leipzig, D-04109 Leipzig, Germany}
\email{laszlo.szekelyhidi@math.uni-leipzig.de}
\fi

\usepackage[active]{srcltx}
\usepackage{ifpdf}
\ifpdf 
    \usepackage[pdftex,     
            plainpages=false,   
            breaklinks=true,    
            colorlinks=true,
            pdftitle=My Document
            pdfauthor=My Good Self
           ]{hyperref} 
\else 
    \usepackage{hyperref}       
\fi

\usepackage{amsfonts,amsmath,esint, xfrac}	
\usepackage{amssymb}
\usepackage{verbatim}
\usepackage{amsopn}
\usepackage[english]{babel}
\usepackage{amsthm}
\usepackage{enumerate}
\usepackage{mathrsfs}	
\usepackage{mathtools}
\usepackage{color}
\usepackage{tikz}
\usepackage{soul}
\usepackage{url}
\usepackage[overload]{empheq}

\ifams

\fi

\ifams
\theoremstyle{definition} \newtheorem{definition}{Definition}[section]
\theoremstyle{definition} \newtheorem{remark}[definition]{Remark}
\theoremstyle{plain} \newtheorem{lemma}[definition]{Lemma}
\theoremstyle{plain} \newtheorem{proposition}[definition]{Proposition}
\theoremstyle{plain} \newtheorem{theorem}[definition]{Theorem}
\theoremstyle{plain} 
\theoremstyle{definition} 
\fi

\newcommand{\R}{\mathbb{R}}

\newcommand{\N}{\mathbb{N}}
\newcommand{\Z}{\mathbb{Z}}
\newcommand{\T}{\mathbb{T}}

\renewcommand{\div}{\textrm{div }}
\newcommand{\e}{\varepsilon}

\newcommand{\supp}{\mathrm{supp} \ }
\newcommand{\ackn}{This research was supported by the ERC Grant Agreement No. 724298.}
\newcommand{\keyw}{Transport equation, non-uniqueness, renormalization, convex integration}

\makeatletter
\def\grd@save@target#1{%
  \def\grd@target{#1}}
\def\grd@save@start#1{%
  \def\grd@start{#1}}
\tikzset{
  grid with coordinates/.style={
    to path={%
      \pgfextra{%
        \edef\grd@@target{(\tikztotarget)}%
        \tikz@scan@one@point\grd@save@target\grd@@target\relax
        \edef\grd@@start{(\tikztostart)}%
        \tikz@scan@one@point\grd@save@start\grd@@start\relax
        \draw[minor help lines] (\tikztostart) grid (\tikztotarget);
        \draw[major help lines] (\tikztostart) grid (\tikztotarget);
        \grd@start
        \pgfmathsetmacro{\grd@xa}{\the\pgf@x/1cm}
        \pgfmathsetmacro{\grd@ya}{\the\pgf@y/1cm}
        \grd@target
        \pgfmathsetmacro{\grd@xb}{\the\pgf@x/1cm}
        \pgfmathsetmacro{\grd@yb}{\the\pgf@y/1cm}
        \pgfmathsetmacro{\grd@xc}{\grd@xa + \pgfkeysvalueof{/tikz/grid with coordinates/major step}}
        \pgfmathsetmacro{\grd@yc}{\grd@ya + \pgfkeysvalueof{/tikz/grid with coordinates/major step}}
        \foreach \x in {\grd@xa,\grd@xc,...,\grd@xb}
        \node[anchor=north] at (\x,\grd@ya) {\pgfmathprintnumber{\x}};
        \foreach \y in {\grd@ya,\grd@yc,...,\grd@yb}
        \node[anchor=east] at (\grd@xa,\y) {\pgfmathprintnumber{\y}};
      }
    }
  },
  minor help lines/.style={
    help lines,
    step=\pgfkeysvalueof{/tikz/grid with coordinates/minor step}
  },
  major help lines/.style={
    help lines,
    line width=\pgfkeysvalueof{/tikz/grid with coordinates/major line width},
    step=\pgfkeysvalueof{/tikz/grid with coordinates/major step}
  },
  grid with coordinates/.cd,
  minor step/.initial=.2,           
  major step/.initial=1,            
  major line width/.initial=1pt,    
}
\makeatother

\begin{document}


\begin{abstract}
We show that there are continuous, $W^{1,p}$ ($p<d-1$), incompressible vector fields for which uniqueness of solutions to the continuity equation fails.
\end{abstract}

\ifams \maketitle \fi

\section{Introduction}

In this paper we consider the continuity equation
\begin{equation}
\label{eq:continuity}
\begin{aligned}
\partial_t \rho  + \div(\rho u) & = 0, \\
\div u & = 0.
\end{aligned}
\end{equation} 
in a $d$-dimensional periodic domain, $d \geq 3$, for a 
%
%
%
time-dependent incompressible vector field $u : [0,1] \times \T^d  \to \R^d$ and an unknown density $\rho: [0,1] \times \T^d \to \R$. Here and in the sequel $\T^d = \R^d/\Z^d$ is the $d$-dimensional flat torus. We will also always assume, without loss of generality, that the time interval is $[0,1]$. We prove in these notes the following theorem.
\begin{theorem}[Non-uniqueness for Sobolev and continuous vector fields]
\label{thm:nonuniq}
Let $\e>0$, $\bar \rho \in C^\infty(\T^d)$ with $\int_{\T^d}\bar\rho\,dx=0$. Then there exist
\begin{equation*}
\rho \in C\big([0,1]; L^1(\T^d)\big), \quad u \in C([0,1] \times \T^d) \cap \bigcap_{1 \leq p < d-1} C\big([0,1]; W^{1,p} (\T^d)\big)
\end{equation*}
such that $(\rho,u)$ is a weak solution to \eqref{eq:continuity} and $\rho(0) \equiv 0$ at $t=0$, $\rho(1) \equiv \bar \rho$ at $t=1$ and
\begin{equation}
\label{eq:bound:l:infty}
\max_{(t,x) \in [0,1] \times \T^d} |u(t,x)| \leq \e.
\end{equation}
\end{theorem}
\noindent By \emph{weak solution} we mean solution in the sense of distributions. 

\bigskip

It is well known that the theory of classical solutions to \eqref{eq:continuity} is closely connected to the ordinary differential equation
\begin{equation}\label{e:ODE}
\begin{split}
\partial_t X(t,x) & = u(t,X(t,x)), \\
X(0,x) & = x
\end{split}	
\end{equation}
via the formula $\rho(t) \mathcal{L}^d = X(t)_\sharp (\rho(0)\mathcal{L}^d)$ or, equivalently, due to the incompressibility, 
\begin{equation}
\label{eq:ode:sln}
\rho(t, X(t,x))=\rho(0,x).
\end{equation}
In particular, for Lipschitz vector fields $u$ the well-posedness theory for \eqref{eq:continuity} follows from the Cauchy-Lipschitz theory for ordinary differential equations applied to \eqref{e:ODE}.

It is in general of great interest to investigate the existence and uniqueness of weak solutions to  the Cauchy problem for \eqref{eq:continuity} in the case of non-smooth vector fields, and the connection to the Lagrangian problem \eqref{e:ODE}-\eqref{eq:ode:sln}. The general question can be formulated as follows. Fix an exponent $r \in [1,\infty]$, denote by $r'$ its dual H\"older, $1/r + 1/r' = 1$, and assume that a vector field
\begin{equation}
\label{eq:class:u}
u \in L^1 \big(0,1; L^r(\T^d) \big)
\end{equation}
is given. What can be said about existence and uniqueness of weak solutions in the class of densities
\begin{equation}
\label{eq:class:rho}
\rho \in L^\infty \big(0,1; L^{r'}(\T^d) \big) \ ?
\end{equation}
The choice of class \eqref{eq:class:rho}  for $\rho$ is motivated by the fact that for classical solutions to \eqref{eq:continuity} every spatial $L^{r'}$ norm is preserved in time. Once \eqref{eq:class:rho} is fixed, the choice of the class \eqref{eq:class:u} for $u$ is natural, since in this way $\rho u \in L^1((0,1) \times \T^d)$ and thus the notion of distributional solution to \eqref{eq:continuity} is well defined.  

While existence of weak solutions can be easily shown under the assumptions \eqref{eq:class:u}-\eqref{eq:class:rho}, the uniqueness question is much harder. In 1989 R. DiPerna and P.L. Lions \cite{DiPerna:1989vo} proved that uniqueness holds in the class \eqref{eq:class:rho} if 
\begin{equation}
\label{eq:dpl}
Du \in L^1 \big ( 0,1; L^{r}(\T^d) \big)
\end{equation}
i.e. if $u$ enjoys Sobolev regularity with exponent $r$. Moreover, in this case, the incompressibility assumption can be relaxed to $\div u \in L^\infty$. In the class of bounded densities the uniqueness result was later  extended by L. Ambrosio \cite{Ambrosio:2004cva} in 2004, for fields $u \in L^1(0,1; BV)$ with $\div u \in L^\infty$ and very recently by S. Bianchini and P. Bonicatto \cite{Bianchini:2017vf} in 2017 in the case of $BV$ nearly incompressible vector fields. 

In all of these results an important additional feature is the connection to a suitable extension to \eqref{e:ODE}, i.e. the link between the Eulerian and the Lagrangian picture. More precisely, under assumption \eqref{eq:dpl}, there exists a unique distributional solution to \eqref{e:ODE} for a.e. $x$, such that $x \mapsto X(t,x)$ is measure preserving for all $t$ (assuming $\div u = 0$): such flow map is called \emph{regular Lagrangian flow} (see \cite{Ambrosio2017} for a general discussion). Then the unique solution to the continuity or the transport equation is given  by \eqref{eq:ode:sln}, as in the smooth case. 

On the other side, several non-uniqueness counterexamples are known, but they mainly concern the case when the field is ``very far'' from being incompressible (e.g. $\div u \notin L^\infty$, see \cite{DiPerna:1989vo}) or the case when no bounds on  one full derivative of $u$ are available (see, for instance, the counterexample in \cite{DiPerna:1989vo} for a field $u \in L^1 (0,1; W^{s,1})$ for every $s<1$, but $u \notin L^1(0,1; W^{1,1})$ or the counterexample in \cite{Depauw:2003wl} for a field  $u \in L^1(\e, 1; BV)$ for every $\e>0$, but $u \notin L^1(0,1; BV)$). In all these counterexamples, however, non-uniqueness for the PDE \eqref{eq:continuity} is a consequence of a \emph{Lagrangian} non-uniqueness for the associated ODE \eqref{e:ODE}. We refer to \cite{ModSze:2017} and to \cite{Ambrosio2017} for a more detailed discussion. 

Very recently, we proved in \cite{ModSze:2017} the analog of Theorem \ref{thm:nonuniq}, for fields and densities in the class
\begin{equation*}
\rho \in C\big([0,1]; L^{r'}(\T^d)\big), \quad u \in C \big( [0,1]; L^r(\T^d) \big) \cap C\big([0,1]; W^{1,p} (\T^d)\big),
\end{equation*}
with
\begin{equation}
\label{eq:range:rp}
r \in (1,\infty), \quad  p \in [1,\infty)
\end{equation}
and
\begin{equation}
\label{eq:tilde:p}
\frac{1}{r'} + \frac{1}{p} > 1 + \frac{1}{d-1}.
\end{equation}
The result in \cite{ModSze:2017} shows that uniqueness can fail even for incompressible, Sobolev vector fields (i.e. fields for which the Lagrangian problem \eqref{e:ODE} is well posed, in the sense of the regular Lagrangian flow), if the integrability exponent $p$ of $Du$ is much lower than the one provided in \eqref{eq:dpl} by DiPerna and Lions' theory, as specified in \eqref{eq:tilde:p}.

The end-point $r=\infty$, corresponding in \eqref{eq:tilde:p} to $p < d-1$, is excluded in \cite{ModSze:2017}. The main result of this notes, namely Theorem \ref{thm:nonuniq}, shows that such end-point case can indeed be reached and, in addition, quite surprisingly, the vector field produced by Theorem \ref{thm:nonuniq} is continuous in time and space, not only bounded. 

We postpone to Section \ref{s:comments:proof} a technical discussion about why the case $r=\infty$ was out of reach in \cite{ModSze:2017} and which new ideas are introduced in these notes to deal with such problem.  

We would like now to briefly comment about the continuity of the vector field $u$ produced by Theorem \ref{thm:nonuniq}. It was observed  by L. Caravenna and G. Crippa \cite{Caravenna:2016kg} that the \emph{boundedness} or the \emph{continuity} of the vector field (in addition to some Sobolev regularity) could play a key role in the uniqueness problem in the  class of integrable densities $\rho \in L^1((0,1) \times \T^d)$. It thus turns out to be a very interesting question to ask if, in fact, boundedness or continuity plus Sobolev regularity are enough to guarantee uniqueness. Theorem \ref{thm:nonuniq} shows that this is not the case, if the integrability of $Du$ is lower than a dimensional threshold (precisely, $d-1$).


The idea that the boundedness or the continuity of $u$ can play a crucial role in the uniqueness problem is confirmed by the fact that the majority of the result concerning existence and uniqueness of the \emph{regular Lagrangian flow} associated to a Sobolev or $BV$ vector field $u$ assume that $u \in L^\infty$ (see, for instance, the recent survey \cite{Ambrosio2017}). 

On a different point of view, it is a classical result (see, for instance, \cite{evans10}) that the boundedness of $u$, even without any further Sobolev regularity, is enough to have uniqueness, if a small viscosity is added to the continuity equation:
\begin{equation}
\label{eq:transport:diffusion}
\partial_t \rho + \div(\rho u) = \nu \Delta \rho, \qquad \nu > 0,
\end{equation}
while in \cite{ModSze:2017} we showed that uniqueness for \eqref{eq:transport:diffusion} can drastically fail is $u$ is Sobolev, but not bounded. 


The result in Theorem \ref{thm:nonuniq} is quite surprisingly, even in comparison with our previous result in \cite{ModSze:2017}. Indeed, for vector fields produced by Theorem \ref{thm:nonuniq} the Lagrangian picture is very well behaved: first, the Sobolev regularity implies the existence and uniqueness of the regular Lagrangian flow. Second, the continuity of the field implies that the trajectories provided by the regular Lagrangian flow are $C^1$ in time (and this was not the case for the fields produced in \cite{ModSze:2017}). Third, the bound \eqref{eq:bound:l:infty} means that the length of each trajectory is at most $\e>0$, i.e. particles almost don't move (and, again, this was not the case for the fields produced in \cite{ModSze:2017}). Observe also  that $\e$ in \eqref{eq:bound:l:infty} depends neither on the length of the time interval $[0,1]$ nor on the $L^1$ distance between the initial and the final datum $\|\rho(1) - \rho(0)\|_{L^1(\T^d)} = \|\bar \rho\|_{L^1(\T^d)}$. Nevertheless uniqueness in the Eulerian world gets completely lost. 

\bigskip

We conclude this introduction observing that Theorem \ref{thm:nonuniq} is an immediate application of the following theorem, whose proof is the topic of all next sections. 
\begin{theorem}
\label{thm:main}
Let $\e>0$. Let $ \rho_0: [0,1] \times \T^d \to \R$, $u_0 : [0,1] \times \T^d \to \R^d$ be smooth with
\begin{equation}
\label{eq:condition:on:rho:0:u:0}
\begin{split}
\int_{\T^d} \rho_0(0,x) dx & = \int_{\T^d} \rho_0(t,x) dx, \\
{\rm div}\, u_0 & = 0,
\end{split}
\end{equation} 
for every $t \in [0,1]$.  Set 
\begin{equation}
\label{eq:def:E}
E := \big\{t \in [0,1] \ : \ \partial_t \rho_0(t) + {\rm div} \,(\rho_0(t) u_0(t)) = 0 \big\}. 
\end{equation}
Then there exist $\rho : [0,1] \times \T^d \to \R$, $u:[0,1] \times \T^d \to \R^d$ such that
\begin{enumerate}[(a)]
\item $\rho$, $u$ have the following regularity:
\begin{equation*}
\rho \in C\Big([0,1]; L^1 (\T^d)\Big), \qquad u \in C\Big([0,1] \times \T^d \Big) \cap \bigcap_{1\leq p < d-1} C\Big([0,1]; W^{1, p}(\T^d) \Big);
\end{equation*}
\item $(\rho, u)$ is a weak solution to \eqref{eq:continuity}; 
\item for every $t \in E$, $\rho(t) = \rho_0(t)$, $u(t) = u_0(t)$;
\item $\rho$ is $\e$-close to $\rho_0$ i.e.
\begin{equation*}
\begin{split}
\max_{t \in [0,1]} \|\rho(t) - \rho_0(t)\|_{L^1(\T^d)} 				& \leq \e. 
\end{split}
\end{equation*}
\end{enumerate}
Condition (d) can be substituted by the following:
\begin{enumerate}[(a)]
\item [(d')] $u$ is $\e$-close to $u_0$ i.e.
\begin{equation*}
\begin{split}
\|u - u_0\|_{C^0([0,1] \times \T^d)} \leq \e.
\end{split}
\end{equation*}
\end{enumerate}
\end{theorem}

\begin{proof}[Proof of Theorem \ref{thm:nonuniq} assuming Theorem \ref{thm:main}]
Let $\chi: [0,1] \to \R$ be such that $\chi \equiv 0$ on $[0,1/4]$, $\chi \equiv 1$ on $[3/4, 1]$. Apply Theorem \ref{thm:main} with $ \rho_0(t,x) := \chi(t) \bar \rho(x)$, $ u_0 = 0$. By (c), $\rho(0) \equiv 0$ at $t = 0$ and $\rho(1) \equiv \bar \rho$ at $t=1$. Moreover, by (d'), $\|u\|_{C^0} \leq \e$. 
\end{proof}

\ifams \subsection*{Acknowledgement} \ackn  \fi

\section{Comments on the proof}
\label{s:comments:proof}

We describe in this section what problems arise when one tries to extend the proof provided in \cite{ModSze:2017} to Theorem \ref{thm:nonuniq}, i.e. to the end-point case $r=\infty$ and which new ideas are introduced to solve such problems. 

\subsection{Sketch of the paper \cite{ModSze:2017}}

We first briefly sketch the proof  provided in \cite{ModSze:2017} for the analog of Theorem \ref{thm:nonuniq} under the conditions \eqref{eq:range:rp}, \eqref{eq:tilde:p} . The proof is based on a convex integration scheme, with both oscillations and concentration playing a key role. More precisely, the density $\rho$ and the field $u$ are defined as limit of a sequence $(\rho_q)_q$, $(u_q)_q$ of smooth approximate solutions to the continuity equation
\begin{equation}
\label{eq:sk:cont:def}
\partial_t \rho_q + \div (\rho_q u_q) = - \div R_q,
\end{equation}
where $R_q$ is a smooth vector field converging strongly to zero
\begin{equation}
\label{eq:sk:resto}
\|R_q\|_{C_t L^1_x} \lesssim \delta_q
\end{equation}
with $\delta_q = 2^{-q}$ and $(\rho_q)_q$, $(u_q)_q$ satisfy
\begin{subequations}
\label{eq:sk:convergences:1}
\begin{align}
\sum_q \|\rho_q - \rho_{q-1}\|_{C_t L^{r'}_x} & < \infty, \label{eq:sk:conv:rho} \\ 
\sum_q \|u_q - u_{q-1}\|_{C_t L^r_x} & < \infty, \label{eq:sk:conv:u}
\end{align}
\end{subequations}
and
\begin{equation}
\label{eq:sk:convergences:2}
\sum_q \|D  u_q - Du_{q-1}\|_{C_t L^p_x} < \infty.
\end{equation}
In this way $\rho, u$ are a weak solution to \eqref{eq:continuity} and, moreover, they have the desired regularity. 

The sequence $(\rho_q, u_q, R_q)$ is constructed recursively. Assuming $(\rho_{q-1}, u_{q-1}, R_{q-1})$ are given, one defines
\begin{equation}
\label{eq:rhoq:uq}
\rho_{q} = \rho_{q-1} + \vartheta_{q}, \qquad u_{q} = u_{q-1} + w_{q},
\end{equation} 
where
%
\begin{equation}
\label{eq:sk:def:theta:w}
\vartheta_{q}(t,x) := F \big(R_{q-1}(t,x)\big) \Theta_{\mu_q}(\lambda_q x), \qquad 
w_q(t,x) := G\big(R_{q-1}(t,x)\big) W_{\mu_q} (\lambda_q x)
\end{equation}
where $\lambda_q$ is an \emph{oscillation parameter} and $\mu_q$ is a \emph{concentration parameter}, suitably chosen at each step of iteration, $F,G$ are nonlinear functions and $\{\Theta_\mu\}_{\mu > 0}$ (resp. $\{W_\mu\}_{\mu>0}$) is a family of \emph{Mikado densities} (resp. \emph{Mikado fields}) (see Proposition \ref{p:mikado} below and in particular estimates \eqref{eq:mikado:est}).

It is  proven in \cite{ModSze:2017} that $\vartheta_q, w_q$ satisfy the following estimates:
\begin{subequations}
\begin{align}
\|\vartheta_q\|_{C_t L^{r'}_x} & \lesssim \|R_{q-1}\|_{C_t L^1_x}^{1/r'}, \label{eq:sk:mikado:1} \\
\|w_q\|_{C_t L^r_x} & \lesssim \|R_{q-1}\|_{C_t L^1_x}^{1/r}, \label{eq:sk:mikado:2} \\
\|\vartheta_q\|_{C_t L^{1}_x} & \lesssim \mu_q^{-\gamma_1}, \label{eq:sk:mikado:4} \\
\|w_q\|_{C_t L^1_x} & \lesssim \mu_q^{-\gamma_2}, \label{eq:sk:mikado:5} \\
\|D w_q\|_{C_t L^p_x} & \lesssim \lambda_q \mu_q^{-\gamma_3}, \label{eq:sk:mikado:3}
\end{align}
\end{subequations}
where 
\begin{equation*}
\gamma_1 = (d-1) \bigg(1 - \frac{1}{r'} \bigg), \quad \gamma_2 = (d-1) \bigg(1 - \frac{1}{r} \bigg), \quad \gamma_3 = (d-1) \Bigg[ \frac{1}{r'} + \frac{1}{p} - \bigg( 1 + \frac{1}{d-1} \bigg) \Bigg].  
\end{equation*}
Notice that $\gamma_1 > 0$ because $r < \infty$ (and thus $r' > 1$), $\gamma_2 >0$ because $r>1$ and $\gamma_3 > 0$ because of \eqref{eq:tilde:p}. Estimates \eqref{eq:sk:mikado:1}-\eqref{eq:sk:mikado:2} together with the inductive assumption \eqref{eq:sk:resto} applied to $R_{q-1}$ guarantee the convergences in \eqref{eq:sk:convergences:1}. Estimate \eqref{eq:sk:mikado:3} guarantees the convergence in \eqref{eq:sk:convergences:2}, provided at each step $\mu_q \gg \lambda_q$. 

A computation then shows that, in order for \eqref{eq:sk:cont:def} to be satisfied, $R_q$ must be defined as
\begin{equation}
\label{eq:sk:error}
-R_q = \div^{-1} \bigg[ \underbrace{\div (\vartheta_q w_q - R_{q-1})}_{\text{quadratic term}} + \underbrace{\partial_t \vartheta_q + \div (\vartheta_q u_{q-1}) + \div(\rho_{q-1} w_q)}_{\text{linear term}}  \bigg].
\end{equation}
In order to prove \eqref{eq:sk:resto}, one first use the oscillation parameter $\lambda_q$ to make the (antidivergence of the) quadratic term small. Then, in order to estimate the linear term, one can use concentration. For instance, for the term $\vartheta_q u_{q-1}$, we can use \eqref{eq:sk:mikado:4}
\begin{equation}
\label{eq:sk:transport}
\big\| \div^{-1} \big( \div (\vartheta_q u_{q-1} ) \big) \|_{C_t L^1_x} 
= \| \vartheta_q u_{q-1}\big\|_{C_t L^1_x} 
\lesssim \|\vartheta_q\|_{C_t L^1_x} \leq \mu_q^{-\gamma_1} \leq \delta_q
\end{equation}
provided $\mu_q$ is chosen large enough. A similar estimate holds for $\partial_t \vartheta_q$, again using \eqref{eq:sk:mikado:4}, while for $\rho_{q-1} w_q$ one must use \eqref{eq:sk:mikado:5}. 

This shows that $R_q$ can be suitably defined in order to satisfy \eqref{eq:sk:error}, thus concluding the proof in \cite{ModSze:2017} for the analog of Theorem \ref{thm:nonuniq} under the assumptions \eqref{eq:range:rp}, \eqref{eq:tilde:p}. Let us now discuss why the above proof does not apply to Theorem \ref{thm:nonuniq}, i.e. to the case $r=\infty$, $r' = 1$. 

\subsection{First Issue}

If $r=\infty$, then estimate \eqref{eq:sk:mikado:2} becomes $\|w_q\|_{C_{tx}} \lesssim 1$ and this is not enough to prove the convergence in \eqref{eq:sk:conv:u}. This issue is solved, modifying the definition of $\rho_q, u_q$ in \eqref{eq:rhoq:uq} as
\begin{equation*}
\rho_q := \rho_{q-1} + \eta_q \vartheta_q, \qquad u_q := u_{q-1} + \frac{1}{\eta_q} w_q
\end{equation*}
and choosing $\eta_q := \|R_{q-1}\|^{-1/2}_{C_t L^1_x}$. In this way, using \eqref{eq:sk:mikado:1}, we get
\begin{equation*}
\|\rho_q - \rho_{q-1}\|_{C_t L^1_x} \lesssim \eta_q \|R_{q-1}\|_{C_t L^1_x} \leq \|R_{q-1}\|_{C_t L^1_x}^{1/2} \leq \delta_{q-1}^{1/2}
\end{equation*}
and 
\begin{equation*}
\|u_q - u_{q-1}\|_{C_{tx}} \lesssim \frac{1}{\eta_q} \leq \|R_{q-1}\|_{C_t L^1_x}^{1/2} \leq \delta_{q-1}^{1/2}
\end{equation*}
so that the convergences in \eqref{eq:sk:convergences:1} still holds, and, moreover, the limit vector field $u = \lim u_q$ is continuous, being the uniform limit of smooth fields. See Section \ref{s:main:prop} and, in particular, estimates \eqref{eq:dist:rho:rho:0}, \eqref{eq:dist:u:u:0}.

\subsection{Second Issue}

The second issue concerns the analysis of the linear term in \eqref{eq:sk:error} and in particular estimate \eqref{eq:sk:transport} and the companion estimate for $\partial_t \vartheta_q$. Indeed, if $r=\infty$ and $r'=1$, then $\gamma_1 = 0$ and thus the concentration paramter $\mu_q$ can not be used in \eqref{eq:sk:transport} to make the linear term smaller than $\delta_q$. 

This issue is solved using the inverse flow map associated to $u_{q-1}$, an idea used in \cite{Buckmaster:2014ty} in the framework of the Euler equation, see also \cite{Isett:2016to}, \cite{Buckmaster:2017uz}. Precisely, one separately considers
\begin{equation}
\label{eq:sk:new:linear}
\text{Linear term in \eqref{eq:sk:error}} = \underbrace{\partial_t \vartheta_q + \div(\vartheta_q u_{q-1})}_{\text{transport term}} + \underbrace{\div (\rho_{q-1}u_q)}_{\text{Nash term}}.
\end{equation}
While for the Nash term an estimate similar to \eqref{eq:sk:transport} still holds, since $\gamma_2 = d-1 >0$, in order to treat the transport term, one modifies the definition of $\vartheta_q$ and $w_q$ as follows. The time interval $[0,1]$ is split into $N$ small intervals $\{I_i\}_i$ of size $1/N$. Denoting by $t_i$ the middle point of each $I_i$, one considers the inverse flow map $\Phi_i$ associated to $u_{q-1}$
\begin{equation*}
\begin{cases}
\partial_t \Phi_i + (u_{q-1} \cdot \nabla) \Phi_i & = 0, \\
\Phi_i(t_i,x) & = x
\end{cases}
\end{equation*}
and a partition of unity $\{\alpha_i\}$ subordinated to the partition $\{I_i\}_i$ of $[0,1]$. The definition in \eqref{eq:sk:def:theta:w} is then modified as follows:
\begin{equation}
\label{eq:sk:new:theta:w}
\begin{split}
\vartheta_{q}(t,x) & :=  F \big(R_{q-1}(t,x)\big) \sum_i \alpha_i(t) \Theta_{\mu_q} \big( \lambda_q \Phi_i(t,x) \big), \\ 
w_q(t,x) & := G\big(R_{q-1}(t,x)\big) \sum_i \alpha_i(t)  W_{\mu_q} \big( \lambda_q \Phi_i(t,x) \big).
\end{split}
\end{equation}
With this new definition, the transport term in \eqref{eq:sk:new:linear} assumes the form
\begin{equation*}
\text{Transport term in \eqref{eq:sk:new:linear}} = \sum_i H_i(t,x) \Theta_{\mu_q} \big( \lambda_q \Phi_i(t,x) \big). 
\end{equation*}
The oscillation parameter $\lambda_q$ can now be used to show that
\begin{equation*}
\div^{-1} \bigg[ \text{Transport term in \eqref{eq:sk:new:linear}} \bigg] \approx \frac{1}{\lambda_q} \lesssim \delta_q.
\end{equation*}
See Section \ref{ss:r:transport}.

\subsection{Third Issue}

The third issue appears because of the new definition \eqref{eq:sk:new:theta:w} of $\vartheta_q, w_q$. Indeed if at some time $t \in [0,1]$ two cutoffs $\alpha_i(t) \neq 0$, $\alpha_{i+1}(t) \neq 0$ are active, then in the quadratic term in \eqref{eq:sk:error} a term of the form 
\begin{equation}
\label{eq:sk:interaction}
\div \bigg[ F(R_{q-1}) G(R_{q-1}) \Theta_{\mu_q} \big( \lambda_q \Phi_i(t,x) \big) W_{\mu_q} \big( \lambda_q \Phi_{i+1}(t,x) \big) \bigg]
\end{equation}
appears, i.e. a non-trivial interaction between a Mikado density and a Mikado field. In general there is no reason why one should be able to find a small antidivergence of such term. The problem can be solved, using, at each step $q$ of the construction, two different oscillation parameters $\lambda_q', \lambda_q''$ and two different concentration parameters $\mu_q', \mu_q''$ with
\begin{equation*}
\lambda_q' \ll \lambda_q'', \qquad \mu_q' \ll \mu_q''
\end{equation*}
and modifying one more time the definition of $\vartheta_q, w_q$ as follows:
\begin{equation*}
\begin{split}
\vartheta_q(t,x) & = F(R_{q-1}(t,x)) \Bigg[  \sum_{i \text{ odd}} \alpha_i(t) \Theta_{\mu_q'} \big(\lambda_q' \Phi_i(t,x) \big) + \sum_{i \text{ even}} \alpha_i(t) \Theta_{\mu_q''} \big(\lambda_q'' \Phi_i(t,x) \big) \Bigg], \\
w_q(t,x) & = F(R_{q-1}(t,x)) \Bigg[  \sum_{i \text{ odd}} \alpha_i(t) W_{\mu_q'} \big(\lambda_q' \Phi_i(t,x) \big) + \sum_{i \text{ even}} \alpha_i(t) W_{\mu_q''} \big(\lambda_q'' \Phi_i(t,x) \big) \Bigg].
\end{split}
\end{equation*}
With this new definition,  the main term in the non-trivial interaction in \eqref{eq:sk:interaction} becomes of the form
\begin{equation}
\label{eq:sk:two:par}
\Theta_{\mu'_q} \big( \lambda_q' \Phi_i(t,x) \big) W_{\mu_q''} \big( \lambda_q'' \Phi_{i+1}(t,x) \big) \ \text{ or } \ \Theta_{\mu''_q} \big( \lambda_q'' \Phi_i(t,x) \big) W_{\mu_q'} \big( \lambda_q' \Phi_{i+1}(t,x) \big),
\end{equation}
i.e. the product of a \emph{fast} oscillating function (with frequency $\lambda_q'$) with a \emph{very fast} oscillating function (with frequency $\lambda_q''$), where one of the two factors (namely $W_{\mu_q'}$ or $W_{\mu_q''}$) is small in $L^1(\T^d)$ because of the concentration mechanism (compare with estimate \eqref{eq:sk:mikado:5}). One can then use an improved H\"older inequality (see Lemma \ref{p:improved:holder}) to show that the  terms in \eqref{eq:sk:two:par} are small in $L^1$ and thus conclude the proof of Theorem \ref{thm:nonuniq}.  See Section \ref{ss:r:quadr} and in particular Lemma \ref{l:rinteraction}.

\section{Technical tools}

In this section we provide some technical tools which will be frequently used in the following.
We start by fixing some notation:
\begin{itemize}
\item $\T^d = \R^d / \Z^d$ is the $d$-dimensional flat torus, $d \geq 3$.
\item If $f(t,x)$ is a smooth function of $t \in [0,1]$ and $x \in \T^d$, we denote by 
	\begin{itemize}
	\item $\|f\|_{C^k}$ the sup norm of $f$ together with the sup norm of all its derivatives in time and space up to order $k$;
	\item $\|f(t)\|_{C^k(\T^d)}$ the sup norm of $f$ together with the sup norm  of all its spatial derivatives up to order $k$ at fixed time $t$;
	\item $\|f(t)\|_{L^p(\T^d)}$ the $L^p$ norm of $f$ in the spatial derivatives, at fixed time $t$. Since we will take always $L^p$ norms in the spatial variable (and never in the time variable), we will also use the shorter notation $\|f(t)\|_{L^p}$ to denote the $L^p$ norm of $f$ in the spatial variable.
	\end{itemize}
\item $C^\infty_0(\T^d)$ is the set of smooth functions on the torus with zero mean value.
\item $\N = \{0,1,2, \dots\}$, $\N^* = \N \setminus \{0\}$.
\item We will use the notation $C(A_1, \dots, A_n)$ to denote a constant which depends only on the numbers $A_1, \dots, A_n$.
\end{itemize}

\subsection{Diffeomorphisms of the flat torus}

We discuss in this section standard properties of diffeomorphisms of the flat torus. Let $\Phi: \R^d \to \R^d$ be a smooth diffeomorphism. We say that $\Phi$ is a \emph{diffeomorphism of $\T^d$}, and we write $\Phi : \T^d \to \T^d$, if 
\begin{equation*}
\Phi(x + k) = \Phi(x) + k, \text{ for every } k \in \Z^d. 
\end{equation*}
We say that a diffeomorphism $\Phi : \T^d \to \T^d$ is \emph{measure-preserving} if $|\det D \Phi(x)| = 1$ for every $x \in \T^d$. Given a diffeomorphism $\Phi$, we will often consider
\begin{enumerate}
\item the derivative $D \Phi : \T^d \to \R^{d \times d}$;
\item the inverse-matrix of the derivative $(D \Phi)^{-1} : \T^d \to \R^{d \times d}$;
\item higher order derivatives of the inverse-matrix of the derivative  $D^k ((D \Phi)^{-1}) : \T^d \to \R^{d(k+2)}$.
\end{enumerate}
Observe that, given a matrix $A \in \R^{d \times d}$, with $|\det A| = 1$, it holds $|A| \geq 1$, where $|A| := \max_{|u| = 1} |A u|$ is the norm of matrix $A$. Therefore if $\Phi$ is a measure-preserving diffeomorphism, then $|D \Phi(x)| \geq 1$ for every $x \in \T^d$ and thus $1 \leq \|D \Phi\|_{C^k}^\alpha \leq \|D \Phi\|_{C^k}^\beta$ for every $0 < \alpha < \beta$. Recall also that for a given invertible matrix $A$, 
\begin{equation*}
A^{-1} = \frac{1}{\det A} ({\rm cof}\, A)^T,
\end{equation*}
where $({\rm cof} \, A)^T$ is transpose of the cofactor matrix of $A$. 

\begin{lemma}
\label{l:diffeo}
Let $\Phi: \T^d \to \T^d$ be a measure-preserving smooth diffeomorphism. Then, for every $k \in \N$,
\begin{equation*}
\|D^k ((D \Phi)^{-1})\|_{C^0} \leq C_k \|D \Phi\|_{C^k}^{d-1},
\end{equation*}
where $C_k$ is a  constant depending only on $k$ (and on the dimension $d$). 
\end{lemma}

\begin{proof}
For any fixed $x \in \T^d$ it holds
\begin{equation}
\label{eq:derivative:phi}
\Big|\big[D \Phi(x)\big]^{-1}\Big| = \bigg| \frac{1}{\det D \Phi(x)} ({\rm cof} \, D \Phi(x))^T \bigg| = \big| {\rm cof} \, D \Phi(x)\big|.
\end{equation}  
The conclusion now follows from the definition of cofactor matrix.
%
%
\end{proof}

\begin{lemma}
\label{l:divergence:diffeo}
Let $G: \T^d \to \R^d$, $g: \T^d \to \R$ be smooth and assume $\div G = g$. Let $\Phi : \T^d \to \T^d$ be a measure-preserving diffeomorphism of the torus. Then 
\begin{equation*}
{\rm div}\, \Big[ (D \Phi)^{-1} G(\Phi) \Big] = g (\Phi).
\end{equation*}
\end{lemma}
\begin{proof}
We show that for every $\varphi \in C^\infty(\T^d)$ it holds
\begin{equation}
\label{eq:divergence:diffeo}
 \int_{\T^d} \varphi \, \div \Big[ (D\Phi)^{-1} G(\Phi) \Big] dx = \int_{\T^d} \varphi \, g(\Phi) dx. 
\end{equation}
Set $\tilde \varphi := \varphi \circ \Phi^{-1}$. It holds
\begin{equation*}
\begin{split}
\int_{\T^d} \varphi \, \div \Big[ (D\Phi)^{-1} G(\Phi) \Big] dx 
& = \int_{\T^d} \tilde \varphi (\Phi) \, \div \Big[ (D\Phi)^{-1} G(\Phi) \Big] dx  \\
& = - \int_{\T^d} \big[ (D \Phi)^T \nabla \tilde \varphi(\Phi) \big] \cdot \big[ (D\Phi)^{-1} G(\Phi) \big] dx \\
& = - \int_{\T^d} \nabla \tilde  \varphi(\Phi) \cdot G( \Phi) dx \\
\text{(changing variable $y = \Phi(x)$)} 
& = - \int_{\T^d} \nabla \tilde \varphi \cdot G \, dy \\
& = \int_{\T^d} \tilde \varphi \, \div G \, dy \\
& = \int_{\T^d} \tilde \varphi \, g \, dy \\
& = \int_{\T^d} \varphi(\Phi^{-1})\, g \, dy \\
\text{(changing variable $x = \Phi^{-1}(y)$)} 
& = \int_{\T^d}  \varphi \, g(\Phi) dx,
\end{split}
\end{equation*}
thus concluding the proof of the lemma. 
\end{proof}

\begin{lemma}
\label{l:diffeo:estimate}
Let $g: \T^d \to \R$ be a smooth function. Let $\Phi: \T^d \to \T^d$ be  a measure-preserving diffeomorphism. Then for every $p \in [1, \infty]$ and $k \in \N$, $k \geq 1$,
\begin{equation*}
\begin{split}
\|g \circ \Phi\|_{L^p(\T^d)} & = \|g\|_{L^p(\T^d)}, \\
\end{split}
\end{equation*}
and
\begin{equation*}
\begin{split}
\|g \circ \Phi\|_{W^{k,p}(\T^d)} & \leq C_k  \|D \Phi\|_{C^{k-1}(\T^d)}^k \|g\|_{W^{k,p}(\T^d)}.
\end{split}
\end{equation*}
\end{lemma}
\noindent The proof is an easy application of the chain rule and  thus it is omitted. 

\subsection{Properties of fast oscillations}
We discuss now some properties of fast oscillating periodic functions. For a given $g : \T^d \to \R$ and $\lambda \in \N^*$, we set
\begin{equation*}
g_\lambda(x) := g(\lambda x).
\end{equation*}
Observe that for every $p \in [1, \infty]$ and $k \in \N$,
\begin{equation}
\label{eq:fast:osc:deriv}
\|D^k g_\lambda\|_{L^{p}(\T^d)} = \lambda^k \|D^k g\|_{L^{p}}.
\end{equation}
Moreover, if $G: \T^d \to \R^d$, $g: \T^d \to \R$ are smooth and $\div G = g$, then
\begin{equation}
\label{eq:fast:osc:div}
{\rm div}\, G_\lambda = \lambda g_\lambda.
\end{equation}

\subsubsection{Improved H\"older inequality}

In the same spirit as in \cite{ModSze:2017} and \cite{Buckmaster:2017wf}, we now prove an improved H\"older inequality for the product of a slow oscillating function with a fast oscillating functions composed with a diffeomorphism. 

\begin{lemma}[Improved H\"older inequality]
\label{p:improved:holder}
Let $f,g: \T^d \to \R$ be smooth functions, $\lambda \in \N^*$ and $\Phi: \T^d \to \T^d$ be a measure-preserving diffeomorphism. Then for every $p \in [1,\infty]$,
\begin{equation}
\label{eq:improved:holder}
\|f g_\lambda\|_{L^p} \leq \|f\|_{L^p} \|g\|_{L^p} + \frac{C_p}{\lambda^{1/p}} \|f\|_{C^1} \|g\|_{L^p}
\end{equation}
and
\begin{equation}
\label{eq:improved:holder:diffeo}
\|f \, (g_\lambda \circ \Phi)\|_{L^p} \leq \|f\|_{L^p} \|g\|_{L^p} + \frac{C_p}{\lambda^{1/p}}\|f\|_{C^1} \|D \Phi\|^{d-1}_{C^0} \|g\|_{L^p}.  
\end{equation}
\end{lemma}
\noindent Here $f \, (g_\lambda \circ \Phi)$ is the function $x \mapsto f(x) g(\lambda \Phi(x))$. 
\begin{proof}
For a proof of \eqref{eq:improved:holder}, see \cite[Lemma 2.1]{ModSze:2017}. Concerning \eqref{eq:improved:holder:diffeo}, we argue as follows. Since $\Phi$ is a measure-preserving diffeomorphism, it holds
\begin{equation*}
\|f \, (g_\lambda \circ \Phi)\|_{L^p} = \|(f \circ \Phi^{-1}) g_{\lambda}\|_{L^p}.
\end{equation*}
Therefore we can apply \eqref{eq:improved:holder} to get
\begin{equation*}
\begin{split}
\|f \, (g_\lambda \circ \Phi)\|_{L^p} & \leq \|f \circ \Phi^{-1}\|_{L^p} \|g\|_{L^p} + \frac{C_p}{\lambda^{1/p}} \|f \circ \Phi^{-1}\|_{C^1} \|g\|_{L^p} \\
\text{(by Lemma \ref{l:diffeo:estimate} and \eqref{eq:fast:osc:deriv})} 
& \leq \|f\|_{L^p} \|g\|_{L^p} + \frac{C_p}{\lambda^{1/p}} \|f\|_{C^1} \|D(\Phi^{-1})\|_{C^0} \|g\|_{L^p} \\
& \leq \|f\|_{L^p} \|g\|_{L^p} + \frac{C_p}{\lambda^{1/p}} \|f\|_{C^1} \|(D\Phi)^{-1}\|_{C^0} \|g\|_{L^p} \\
\text{(by Lemma \ref{l:diffeo})} & \leq \|f\|_{L^p} \|g\|_{L^p} + \frac{C_p}{\lambda^{1/p}} \|f\|_{C^1} \|D\Phi\|^{d-1}_{C^0} \|g\|_{L^p}.  
\end{split} 
\end{equation*}
\end{proof}

\subsubsection{Antidivergence operators}

In this section we introduce two antidivergence operators, a standard and an improved one, in the same spirit as in \cite{ModSze:2017}. 

For $f \in C^\infty_0(\T^d)$ there exists a unique $u \in C^\infty_0(\T^d)$ such that $\Delta u = f$. The operator $\Delta^{-1}: C^\infty_0(\T^d) \to C^\infty_0(\T^d)$ is thus well defined. We define the \emph{standard antidivergence operator} as $\nabla \Delta^{-1}: C^\infty_0(\T^d) \to C^\infty(\T^d; \R^d)$. It clearly satisfies $\div (\nabla \Delta^{-1} f) = f$. 

\begin{lemma}
\label{l:antidiv}
For every $k \in \N$ and $p \in [1, \infty]$, the \emph{standard} antidivergence operator satisfies the bounds \begin{equation}
\label{eq:stand:antidiv}
\big\|D^k (\nabla \Delta^{-1} g) \big\|_{L^p} \leq C_{k,p} \|D^k g\|_{L^p}.
\end{equation}
\end{lemma}
\noindent For the proof, see \cite[Lemma 2.2]{ModSze:2017} .

We now use introduce an improved antidivergence operator. 
\begin{lemma}
\label{p:antidiv:transport}
Let $f, g: \T^d \to \R$ be smooth function with
\begin{equation*}
\fint g = 0.
\end{equation*}
Let $\lambda \in \N^*$ and $\Phi: \T^d \to \T^d$ be a smooth, measure-preserving diffeomorphism. Then there exists a smooth vector field $u : \T^d \to \R^d$ so that
\begin{equation}
\label{eq:antidiv:transport}
\div u = f \, (g_\lambda \circ \Phi) - \fint f \, (g_\lambda \circ \Phi) 
\end{equation}
and for every $k \in \N$, $p \in [1, \infty]$, 
\begin{equation}
\label{eq:antidiv:transport:estimate}
\|u\|_{W^{k,p}} \leq C_{k,p} \lambda^{k-1} \|f\|_{C^{k+1}} \|D\Phi\|_{C^{k+1}}^{d-1+k} \|g\|_{W^{k,p}}. 
\end{equation}
\end{lemma}
We will use the notation
\begin{equation*}
u := \mathcal{R} \Bigg( f \, ( g_\lambda \circ \Phi  ) - \fint f \, (g_\lambda \circ \Phi)  \Bigg).
\end{equation*}

\begin{remark}
The same result holds if $f,g$ are vector fields and we want to solve 
\begin{equation*}
\div u = f \cdot (g_\lambda \circ \Phi) - \fint f \cdot (g_\lambda \circ \Phi) ,
\end{equation*}
where $\cdot$ denotes the scalar product. 
\end{remark}

\begin{proof}
Since $g$ has zero mean value, we can define 
\begin{equation}
\label{eq:def:G}
G := \nabla \Delta^{-1} g
\end{equation}
and
\begin{equation*}
u := \frac{1}{\lambda} \Bigg\{ f (D\Phi)^{-1} \, (G_\lambda \circ \Phi) -  \nabla \Delta^{-1} \bigg[ \nabla f  \cdot \Big( (D \Phi)^{-1} \, (G_\lambda \circ \Phi)  \Big)  - \fint \nabla f \cdot  \Big( (D \Phi)^{-1} \, (G_\lambda \circ \Phi) \Big)     \bigg] \Bigg\}
\end{equation*}
Let us first check that $u$ satisfies \eqref{eq:antidiv:transport}. It holds
\begin{equation*}
\begin{split}
\div u & = \frac{1}{\lambda} \Bigg\{ f \div \Big[  (D \Phi)^{-1} \, (G_\lambda \circ \Phi) \Big] +  \nabla f \cdot (D \Phi)^{-1} \, (G_\lambda \circ \Phi)  \\
& \qquad \quad - \nabla f \cdot (D \Phi)^{-1} \, (G_\lambda \circ \Phi) + \fint \nabla f \cdot  (D \Phi)^{-1} \, (G_\lambda \circ \Phi)   dx \Bigg\} \\
& = \frac{1}{\lambda}  f \div \Big[  (D \Phi)^{-1} \, (G_\lambda \circ \Phi) \Big]  + \frac{1}{\lambda }\fint \nabla f \cdot  (D \Phi)^{-1} \, (G_\lambda \circ \Phi)   dx \\
\text{(integrating by parts)}
& = \frac{1}{\lambda}  f \div \Big[  (D \Phi)^{-1} \, (G_\lambda \circ \Phi) \Big] - \frac{1}{\lambda} \fint f \, \div \Big[(D \Phi)^{-1} \, (G_\lambda \circ \Phi) \Big] dx \\
\text{(by Lemma \ref{l:divergence:diffeo})}
& = \frac{1}{\lambda} f \, (\div G_\lambda) \circ \Phi - \frac{1}{\lambda} \fint f \, (\div G_\lambda) \circ \Phi \, dx \\
\text{(by \eqref{eq:fast:osc:div})}
& = f \, (g_\lambda \circ \Phi) - \fint f\, (g_\lambda \circ \Phi) dx.
\end{split}
\end{equation*}
We prove now that \eqref{eq:antidiv:transport:estimate} holds. We can write $u = \lambda^{-1} ( A - \nabla \Delta^{-1} B)$, where
\begin{equation*}
\begin{split}
A & :=  f \, (D\Phi)^{-1} \, (G_\lambda \circ \Phi), \\
B & := \nabla f \cdot (D \Phi)^{-1} \, (G_\lambda \circ \Phi)  - \fint \nabla f \cdot  (D \Phi)^{-1} \, (G_\lambda \circ \Phi)   dx.
\end{split}
\end{equation*}
Let us estimate $A$:
\begin{equation*}
\begin{split}
\|A\|_{W^{k,p}} 
& \leq \|f\|_{C^k} \|(D \Phi)^{-1}\|_{C^k} \|G_\lambda \circ \Phi\|_{W^{k,p}} \\
\text{(by Lemma \ref{l:diffeo})} 
& \leq \|f\|_{C^k} \|D \Phi\|^{d-1}_{C^k} \|G_\lambda \circ \Phi\|_{W^{k,p}}.
\end{split}
\end{equation*}
Similarly, for $B$:
\begin{equation*}
\begin{split}
\|B\|_{W^{k,p}} 
& \leq 2 \|f\|_{C^{k+1}} \|(D \Phi)^{-1}\|_{C^k} \|G_\lambda \circ \Phi\|_{W^{k,p}} \\
\text{(by Lemma \ref{l:diffeo})} 
& \leq 2 \|f\|_{C^{k+1}} \|D \Phi\|^{d-1}_{C^k} \|G_\lambda \circ \Phi\|_{W^{k,p}}
\end{split}
\end{equation*}
and thus
\begin{equation*}
\begin{split}
\|u\|_{W^{k,p}} 
& \leq \frac{1}{\lambda} \Big\{  \|A\|_{W^{k,p}} + \|\nabla \Delta^{-1} B\|_{W^{k,p}} \Big\} \\
\text{(by Lemma \ref{l:antidiv})}
& \leq \frac{1}{\lambda} \Big\{ \|A\|_{W^{k,p}} + \|B\|_{W^{k,p}} \Big\} \\ 
& \leq \frac{3}{\lambda}    \|f\|_{C^{k+1}} \|D \Phi\|^{d-1}_{C^k} \|G_\lambda \circ \Phi\|_{W^{k,p}}  \\
\text{(by Lemma \ref{l:diffeo:estimate})}
& \leq \frac{C_k}{\lambda} \|f\|_{C^{k+1}} \|D \Phi\|^{d-1+k}_{C^k} \|G_\lambda\|_{W^{k,p}} \\
\text{(by \eqref{eq:fast:osc:deriv})} 
& \leq C_k \lambda^{k-1} \|f\|_{C^{k+1}} \|D \Phi\|^{d-1+k}_{C^k} \|G\|_{W^{k,p}} \\ 
\text{(by Lemma \ref{l:antidiv} and \eqref{eq:def:G})}
& \leq C_k \lambda^{k-1} \|f\|_{C^{k+1}} \|D \Phi\|^{d-1+k}_{C^k} \|g\|_{W^{k,p}}.
\end{split}
\end{equation*}

\end{proof}

\begin{remark}
\label{rmk:antidiv:time}
In Lemma \ref{p:antidiv:transport}, if $f,g,\Phi$ are smooth functions of $(t,x)$, $t \in [0,1]$, $x \in \T^d$ and at each time $t \in [0,1]$, they satisfy the assumptions of Lemma \ref{p:antidiv:transport}, then we can apply $\mathcal{R}$  at each time and define a time-dependent vector field $u(t, \cdot)$ satisfying \eqref{eq:antidiv:transport} and \eqref{eq:antidiv:transport:estimate}. Moreover $u$ turns out to be a smooth function of $(t,x)$. 
\end{remark}

%
%

\subsubsection{Mean value and fast oscillations}

In this section we prodide an estimate on the mean value of the product of a slow oscillating function with a fast oscillating function composed with a diffeomorphism. 

\begin{lemma}
\label{l:mean:value}
Let $f, g : \T^d \to \R$, with $\fint_{\T^d} g = 0$. Let $\lambda \in \N^*$ and $\Phi: \T^d \to \T^d$ be a measure-preserving diffeomorphism. Then
\begin{equation}
\label{eq:mean:value}
\bigg| \fint_{\T^d} f g_\lambda  dx \bigg| \leq \frac{\sqrt{d} \|f\|_{C^1} \|g\|_{L^1}}{\lambda}
\end{equation}
and
\begin{equation}
\label{eq:mean:value:diffeo}
\bigg| \fint_{\T^d} f \, (g_\lambda \circ \Phi) dx \bigg| \leq \frac{\sqrt{d} \|f\|_{C^1} \|D \Phi\|_{C^0}^{d-1} \|g\|_{L^1}}{\lambda}. 
\end{equation}
\end{lemma}
\begin{proof}
For a proof of \eqref{eq:mean:value}, see \cite[Lemma 2.6]{ModSze:2017}. The proof of \eqref{eq:mean:value:diffeo} follows from \eqref{eq:mean:value}, observing that
\begin{equation*}
\fint f(x) g (\lambda \Phi(x)) dx = \fint f(\Phi^{-1}(y)) g(\lambda y) dy. 
\end{equation*}
\end{proof}


\section{Statement of the main proposition and proof of Theorem \ref{thm:main}}
\label{s:main:prop}

We assume without loss of generality $\T^d$ is the periodic extension of the unit cube $[0,1]^d$. 
The following proposition contains the key facts used to prove Theorem \ref{thm:main}. Let us first introduce the \emph{continuity-defect} equation:
\begin{equation}
\label{eq:cont:reyn}
\left\{
\begin{split}
\partial_t \rho + \div(\rho u) 	& = - \div R, \\
\div u 								& = 0.
\end{split}
\right.
\end{equation}
We will call $R$ the \emph{defect field}. 

\begin{proposition}
\label{p:main}
There exists a constant $M>0$ such that the following holds. Let $p \in [1, d-1)$, $\eta, \delta> 0$ and let $(\rho_0, u_0, R_0)$ be a smooth solution of the continuity-defect equation \eqref{eq:cont:reyn}. Then there exists another smooth solution $(\rho_1, u_1, R_1)$ of \eqref{eq:cont:reyn} such that for every $t \in [0,1]$,
\begin{subequations}
\begin{align}
\|\rho_1(t) - \rho_0(t)\|_{L^1(\T^d)} & \leq M \eta \|R_0(t)\|_{L^1(\T^d)}, \label{eq:dist:rho:stat} \\
\|u_1(t) - u_0(t)\|_{C^0(\T^d)} & \leq M \eta^{-1} \label{eq:dist:u:1:stat} \\
\|u_1(t) - u_0(t)\|_{W^{1,p}(\T^d)} & \leq \delta,                   \label{eq:dist:u:2:stat} \\
\|R_1(t)\|_{L^1(\T^d)} 						& \leq \delta,					 \label{eq:reyn:stat}
\end{align}
\end{subequations}
and, moreover, if at some time $t \in [0,1]$, $R_0(t) =0$, then
\begin{equation*}
\rho_1(t) - \rho_0(t) = u_1(t) - u_0(t) = R_1(t) = 0.  
\end{equation*}
\end{proposition}

\begin{proof}[Proof of Theorem \ref{thm:main} assuming Proposition \ref{p:main}]

For $\rho_0, u_0$ in the statement of Theorem \ref{thm:main}, define 
\begin{equation*}
R_0(t) := - \nabla \Delta^{-1} \Big(\partial_t \rho_0(t) + \div (\rho_0(t) u_0(t)) \Big).
\end{equation*}
By \eqref{eq:condition:on:rho:0:u:0}, $R_0$ is well defined, it is smooth and $(\rho_0, u_0, R_0)$ solve the continuity-defect equation. 

Let $(p_q)_{q \in \N}$ be a fixed increasing sequence of real numbers such that $p_q \to d-1$ as $q \to \infty$. Let also $(\eta_q)_{q \in \N}$, $(\delta_q)_{q \in \N}$ be two sequence of positive real numbers, which will be fixed later. Starting from $(\rho_0, u_0, R_0)$, we can recursively apply Proposition \ref{p:main} to obtain a sequence $(\rho_q, u_q, R_q)_{q \in \N}$ of smooth solutions to the continuity-defect equation such that
\begin{subequations}
\begin{align}
\|\rho_{q+1}(t) - \rho_q(t)\|_{L^1(\T^d)} & \leq M \eta_q \|R_q(t)\|_{L^1(\T^d)}, \label{eq:dist:rho:stat:ind} \\
\|u_{q+1}(t) - u_q(t)\|_{C^0(\T^d)} & \leq M \eta_q^{-1} \label{eq:dist:u:1:stat:ind} \\
\|u_{q+1}(t) - u_q(t)\|_{W^{1,p_q}(\T^d)} & \leq \delta_q,                   \label{eq:dist:u:2:stat:ind} \\
\|R_{q+1}(t)\|_{L^1(\T^d)} 						& \leq \delta_q,					 \label{eq:reyn:stat:ind}
\end{align}
\end{subequations}
for all times $t \in [0,1]$ and
\begin{equation*}
\rho_{q+1}(t) = \rho_q(t), \quad u_{q+1}(t) = u_q(t), \quad R_{q+1}(t) = 0, 
\end{equation*}
for all times $t$ such that $R_q(t) = 0$. Therefore, by induction, we get from \eqref{eq:dist:rho:stat:ind} and \eqref{eq:reyn:stat:ind} that for all $t \in [0,1]$ and all $q \in \N$, 
\begin{equation}
\|\rho_{q+1}(t) - \rho_q(t)\|_{L^1(\T^d)} \leq M \eta_q \delta_{q-1}, \label{eq:dist:rho:stat:ind:2}
\end{equation}
where we set $\delta_{-1} := \max_{t \in [0,1]} \|R_0(t)\|_{L^1}$ and, moreover, 
\begin{equation}
\label{eq:E}
\rho_{q+1}(t) = \rho_q(t), \quad u_{q+1}(t) = u_q(t) \quad \text{ for all } t \in E,
\end{equation}
where $E$  was defined in \eqref{eq:def:E}. We now choose $(\delta_q)_{q \in \N}$ so that
\begin{equation*}
\sum_{q=-1}^{+\infty} \delta_q < \infty, \qquad \sum_{q=-1}^{+\infty} \delta_q^{1/2} < \infty
\end{equation*}
and 
\begin{equation*}
\eta_q := \sigma \delta_{q-1}^{-1/2}
\end{equation*}
for $q \in \N$, where $\sigma > 0$ is a positive number,  to be defined later. From \eqref{eq:dist:rho:stat:ind:2} we get, for all $t \in [0,1]$, 
\begin{equation}
\label{eq:dist:rho:rho:0}
\sum_{q = 0}^{+\infty} \|\rho_{q+1}(t) - \rho_q(t)\|_{L^1} \leq M \sum_{q=0}^{+\infty}  \eta_q \delta_{q-1} = M \sigma \sum_{q=0}^{+\infty} \delta_{q-1}^{1/2} < \infty
\end{equation}
and thus there exists $\rho \in C([0,1]; L^1(\T^d))$ so that $\rho_q \to \rho$ in $C([0,1]; L^1(\T^d))$. Similarly, using \eqref{eq:dist:u:1:stat:ind}, for all $t \in [0,1]$, 
\begin{equation}
\label{eq:dist:u:u:0}
\sum_{q=0}^{+\infty} \|u_{q+1}(t) - u_q(t)\|_{C^0} \leq M \sum_{q=0}^{+\infty} \eta_q^{-1} = M \sigma^{-1} \sum_{q=0}^{+\infty} \delta_{q-1}^{1/2} < \infty
\end{equation}
and thus there exists $u \in C([0,1] \times \T^d; \R^d)$ so that $u_q \to u$ uniformly. It follows now from \eqref{eq:reyn:stat:ind} that $\rho, u$ solve \eqref{eq:continuity}. 

To prove that $u \in \bigcap_{1 \leq p < d-1} C_t W_x^{1, p}$, fix $p \in [1, d-1)$. There is $q^*$ so that $p_{q} > p$ for every $q > q^*$. We now have, for all $t \in [0,1]$,
\begin{equation*}
\begin{split}
\sum_{q=0}^{+\infty} \|u_{q+1}(t) - u_q(t)\|_{W^{1, p}} 
& = \sum_{q=0}^{q^*} \|u_{q+1}(t) - u_q(t)\|_{W^{1, p}} + \sum_{q=q^*+1}^{+\infty} \|u_{q+1}(t) - u_q(t)\|_{W^{1, p}} \\
\text{(since $p < p_q$ for $q > q^*$)}
& \leq \sum_{q=0}^{q^*} \|u_{q+1}(t) - u_q(t)\|_{W^{1, p}} + \sum_{q=q^*+1}^{+\infty} \|u_{q+1}(t) - u_q(t)\|_{W^{1, p_q}} \\
\text{(by \eqref{eq:dist:u:2:stat:ind})}
& \leq \sum_{q=0}^{q^*} \|u_{q+1}(t) - u_q(t)\|_{W^{1, p}} + \sum_{q=q^*+1}^{+\infty} \delta_q \\
& < \infty,
\end{split}
\end{equation*}
thus proving that $u \in C([0,1]; W^{1, p}(\T^d))$. This concludes the proof of parts (a), (b) in the statement of Theorem \ref{thm:main}.

It follows from \eqref{eq:E} that $\rho(t) = \rho_0(t)$ and $u(t) = u_0(t)$, whenever $t \in E$, and thus part (c) is also proven. To prove (d), we observe that, from \eqref{eq:dist:rho:rho:0}, for all $t \in [0,1]$,
\begin{equation*}
\|\rho(t) - \rho_0(t)\|_{L^1} \leq \sum_{q = 0}^{+\infty} \|\rho_{q+1}(t) - \rho_q(t)\|_{L^1} \leq M \sigma \sum_{q=0}^\infty \delta_{q-1}^{1/2} 
\end{equation*}
and thus (d) follows choosing
\begin{equation*}
\sigma := \frac{\e}{M \sum_{q=0}^{+\infty} \delta_{q-1}^{1/2}}.
\end{equation*}
Alternatively, to achieve (d'), we observe that, from \eqref{eq:dist:u:u:0}, for all $t \in [0,1]$,
\begin{equation*}
\|u - u_0\|_{C^0} \leq M \sigma^{-1} \sum_{q=0}^{+\infty} \delta_{q-1}^{1/2}
\end{equation*}
and thus (d)' follows choosing
\begin{equation*}
\sigma := \frac{M \sum_{q=0}^{+\infty} \delta_{q-1}^{1/2}}{\e}.  
\end{equation*}

\end{proof}


\section{The perturbations}
\label{s:perturbations}

In this and the next two sections we prove Proposition \ref{p:main}.  In particular in this section we fix the constant $M$ in the statement of the proposition, we define the functions $\rho_1$ and $u_1$ and we estimate them. In Section \ref{s:reynolds} we define $R_1$ and we estimate it. In Section \ref{s:proof:prop} we conclude the proof of Proposition \ref{p:main}.

\subsection{Mikado fields and Mikado densities}

We recall the following proposition from \cite{ModSze:2017}. 
%


\begin{proposition}
\label{p:mikado}
Let $a, b \in \R$ with 
\begin{equation}
\label{eq:exponent}
a+b = d-1.
\end{equation}
For every $\mu>2d$ and $j=1, \dots, d$ there exist a \emph{Mikado density} $\Theta_{\mu}^{j} : \T^d \to \R$ and a \emph{Mikado field} $W_\mu^j :\T^d \to \R^d$ with the following properties.
\begin{enumerate}[(a)]

\item It holds
\begin{equation}
\label{eq:mikado:eq}
\begin{cases}
\div W_\mu^j & = 0, \\
\div (\Theta_\mu^j W_\mu^j) & = 0, \\
\fint_{\T^d} \Theta_{\mu}^j = \fint_{\T^d} W_\mu^j & = 0, \\
\fint_{\T^d} \Theta_\mu^j W_\mu^j & = e_j, \\
\end{cases}
\end{equation}
where $\{e_j\}_{j=1,\dots,d}$ is the standard basis in $\R^d$.

\item For every $k \in \N$ and  $r \in [1,\infty]$
\begin{equation}
\label{eq:mikado:est}
\begin{split}
\|D^k \Theta_\mu^j\|_{L^r(\T^d)} 			& \leq M_{k}\, \mu^{a + k - (d-1) /r}, \\
\|D^k W_\mu^j\|_{L^{r}(\T^d)}				& \leq M_{k}\, \mu^{b + k - (d-1)/r}, \\
\end{split}
\end{equation}
where $M_{k}$ is a constant which depends only on $k$, but not on $r$ and $\mu$. 

\item For $j \neq k$, $\supp \Theta_\mu^j = \supp W_\mu^j$ and $\supp \Theta_\mu^j \cap \supp W_{\mu}^k = \emptyset$.

\end{enumerate}
\end{proposition}

We now define the constant $M$ in the statement of Proposition \ref{p:main} as 
\begin{equation}
\label{eq:M}
M := 4d \max \Big\{ M_{0}, \, M_{0}^2, \, M_0 + M_1 \Big\}.
\end{equation}
and we choose
\begin{equation}
\label{eq:choice:ab}
a := d-1, \qquad b := 0.
\end{equation} 
in Proposition \ref{p:mikado}. In this way for each direction $j =1,\dots, d$, we obtain a family of Mikado densities $\{\Theta_\mu^j\}_{\mu > 2d}$ and fields $\{W_\mu^j\}_{\mu>2d}$, obeying the following estimates:
\begin{equation}
\label{eq:mikado:est:1}
\begin{split}
\sum_{j=1}^d \|\Theta_\mu^j\|_{L^1(\T^d)}, \  \sum_{j=1}^d \|W_\mu^j\|_{L^{\infty}(\T^d)}, \ \sum_{j=1}^d	\|\Theta_\mu^j W_\mu^j\|_{L^1(\T^d)}		& \leq \frac{M}{4}, \\
\end{split}
\end{equation}
and
\begin{equation}
\label{eq:mikado:est:2}
\begin{split}
\|W_\mu^j\|_{L^{1}(\T^d)} \leq M \mu^{-(d-1)}, \qquad \|W_\mu^j\|_{W^{1, p}} 	\leq M \mu^{1-(d-1)/p}.
\end{split}
\end{equation}
and
\begin{equation}
\label{eq:mikado:est:3}
\|\Theta_\mu^j\|_{C^1} \leq M \mu^d, \qquad \|W_\mu^j\|_{C^1} \leq M \mu.
\end{equation}

\subsection{Definition of the perturbations}
\label{ss:def:pert}

We are now in a position to define $\rho_1$, $u_1$. The constant $M$ has already been fixed in \eqref{eq:M}. Let thus $p \in [1, d-1)$, $\eta, \delta>0$ and $(\rho_0, u_0, R_0)$ be a smooth solution to the continuity-defect equation \eqref{eq:cont:reyn}. 

Let
\begin{equation*}
\begin{aligned}
& \tau \in 1/\N^* && \text{ ``time scale''}, \\
& \lambda', \lambda'' \in \N && \text{ ``oscillation'' } \\ 
& \mu', \mu'' > 2d && \text{ ``concentration'' } 
\end{aligned}
\end{equation*}
be parameters, which will be fixed later. Set
\begin{equation*}
N := 1/\tau \in \N^*. 
\end{equation*}
For every $i=1,2,\dots, N$, let $I_i := [i\tau, (i+1)\tau]$ and let $t_i := (i+1/2)\tau$ be the midpoint of $I_i$. Consider a partition of unity $\{\alpha_i\}_{i=1,\dots, N}$ subordinate to the family of intervals $\{I_i\}_{i=1,\dots, N}$. More precisely, for every $i=1,\dots, N$, $\alpha_i \in C^\infty([0,1])$ and 
\begin{itemize}
\item $\supp \alpha_i \in [(i-1/3)\tau, (i+1+1/3)\tau]$;
\item $\alpha_i (t) \in [0,1]$ for every $t \in [0,1]$;
\item $\sum_{i=1}^N \alpha_i^2(t) = 1$ for every $t \in [0,1]$.
\end{itemize}
Notice that for every time $t \in [0,1]$ there is at most one odd index $i_1$ and one even index $i_2$ so that $\alpha_i(t) = 0$ for every $i \neq i_1, i_2$. For every $i=1,\dots, N$, let $\Phi_i : [0,1] \times \T^d \to \T^d$ be the solution to
\begin{equation}
\label{eq:inverse:flow}
\begin{cases}
\partial_t \Phi_i + (u_0 \cdot \nabla) \Phi_i & = 0, \\
\Phi_i(t_i,x) & = x,
\end{cases}
\end{equation}
i.e. the inverse flow map associated to the vector field $u_0$, starting at time $t_i$. Notice that, for fixed $t$, $\Phi_i(t) : \T^d \to \T^d$ is a measure-preserving diffeomorphism.

We denote by $R_{0,j}$ the components of $R_0$, i.e.
\begin{equation*}
R_0(t,x) := \sum_{j=1}^d R_{0,j}(t,x) e_j.
\end{equation*}
Let also $\psi: [0,1] \to \R$ be a smooth function such that $\psi(t) \in [0,1]$ for every $t \in [0,1]$ and 
\begin{equation}
\label{eq:psi}
\psi(t) =
\begin{cases}
0, & \text{if } \|R_0(t)\|_{L^1(\T^d)} \leq \delta/8, \\
1, & \text{if } \|R_0(t)\|_{L^1(\T^d)} \geq \delta/4. 
\end{cases}
\end{equation}

\bigskip

We set
\begin{equation*}
\rho_1 := \rho_0 + \vartheta + \vartheta_c , \qquad u_1 :=  u_0 + w,
\end{equation*}
where $\vartheta, \vartheta_c, w$ are defined as follows. First of all, let $\Theta_\mu^j$, $W_\mu^j$, $j=1,\dots, d$, be the family (depending on $\mu$) of Mikado densities and fields provided by Proposition \ref{p:mikado}, with $a,b$ chosen as in \eqref{eq:choice:ab}. We set
\begin{equation}
\label{eq:perturbation}
\begin{split}
\vartheta (t,x) 	& := \eta \, \psi(t) \Bigg\{ \sum_{\substack{i=1 \\ i \text{ odd}}}^N \alpha_i(t) \sum_{j=1}^d R_{0,j}(t,x) \Theta_{\mu'}^j\big(\lambda' \Phi_i(t,x)\big) 
\\
& \qquad \qquad 
+  \sum_{\substack{i=1 \\ i \text{ even}}}^N \alpha_i(t) \sum_{j=1}^d R_{0,j}(t,x) \Theta_{\mu''}^j\big(\lambda'' \Phi_i(t,x)\big) \Bigg\}, \\
w(t,x) 				& := \frac{\psi(t)}{\eta} \Bigg\{ \sum_{\substack{i=1 \\ i \text{ odd}}}^N \alpha_i(t) \sum_{j=1}^d (D\Phi_i(t,x))^{-1}  W_{\mu'}^j \big(\lambda' \Phi_i(t,x)\big) \\
& \qquad \qquad + \sum_{\substack{i=1 \\ i \text{ even}}}^N \alpha_i(t) \sum_{j=1}^d (D\Phi_i(t,x))^{-1}  W_{\mu''}^j \big(\lambda'' \Phi_i(t,x)\big) \Bigg\}, \\
\vartheta_c(t)		& := - \fint_{\T^d} \vartheta(t,x) dx.
\end{split}
\end{equation}
The factor $(D \Phi_i(t,x))^{-1}$ is the inverse matrix of $D \Phi_i(t,x)$. 
Observe that for fixed $t_0 \in [0,1]$, there are at most one odd index $i_1$ and one even index $i_2$ so that $\alpha_i(t) =0$ if $i \neq i_1, i_2$ and $t$ is close enough to $t_0$ (say, $|t - t_0| \leq 2\tau/3$). Therefore for such times $t$ we can write
\begin{equation}
\label{eq:only:two:cutoff}
\begin{split}
\vartheta(t) & = \eta \, \psi(t) \Bigg\{ \alpha_{i_1}(t) \sum_{j=1}^d R_{0,j}(t) \Theta_{\mu'}^j\big(\lambda' \Phi_{i_1}(t) \big) 
\\
& \qquad \qquad \qquad 
+ \alpha_{i_2}(t) \sum_{j=1}^d R_{0,j}(t) \Theta_{\mu''}^j\big(\lambda'' \Phi_{i_2}(t) \big) \Bigg\} \\
w(t) & = \frac{\psi(t)}{\eta} \Bigg\{ \alpha_{i_1}(t) \sum_{j=1}^d (D \Phi_{i_1}(t))^{-1} W_{\mu'}^j \big( \lambda' \Phi_{i_1}(t) \big) 
\\
& \qquad \qquad \qquad 
+ \alpha_{i_2}(t)  \sum_{j=1}^d (D \Phi_{i_2}(t))^{-1} W_{\mu''}^j \big( \lambda'' \Phi_{i_2}(t) \big) \Bigg\},
\end{split}
\end{equation}
Notice that $\vartheta_0$ and $w$ are smooth functions. Notice also that $\vartheta + \vartheta_c$ has zero mean value in $\T^d$ at each time $t$. Finally observe that $w$ is a sum of terms of the form $(D \Phi)^{-1} (G \circ \Phi)$, with
\begin{equation*}
\Phi = \Phi_i(t), \qquad G = (W_{\mu'}^j)_{\lambda'} \text{ or } G = (W_{\mu''}^j)_{\lambda''}. 
\end{equation*}
Since $\div (W_{\mu})_\lambda = 0$ for every $\mu, \lambda$ (see Proposition \ref{p:mikado}), we get from Lemma \ref{l:divergence:diffeo} that
 each one of these terms is divergence free and thus $\div w = 0$. Therefore 
%
\begin{equation*}
\div u_1 = \div u_ 0 + \div w = 0.
\end{equation*}

\begin{remark}
\label{rmk:zero:if:zero}
Observe that, thanks to the cutoff in time $\psi$, if $R_0(t) \equiv 0$, then $$\vartheta(t) = \vartheta_c(t) = w(t)  \equiv 0.$$ 
\end{remark}

\subsection{Estimates on the perturbation}\label{ss:estimates-perturbation}

In this section we estimate $\vartheta$, $\vartheta_c$, $w$.
%
%

\begin{lemma}[$L^1$-norm of $\vartheta$]
\label{l:lp:vartheta}
For every time $t \in [0,1]$,
\begin{equation*}
\begin{split}
\|\vartheta(t)\|_{L^1(\T^d)} & \leq \frac{M\eta}{2} \|R_0(t)\|_{L^1(\T^d)} + C\Big(M, \eta, \|R_0\|_{C^1}, \max_{i=1,\dots, N} \|D \Phi_i\|_{C^0} \Big) \bigg(\frac{1}{\lambda'} + \frac{1}{\lambda''} \bigg).
\end{split}
\end{equation*}
\end{lemma}

\begin{proof}
Since we have to estimate $\|\vartheta(t)\|_{L^1(\T^d)}$ for every fixed time $t$, we can assume that $\vartheta(t)$ has the form \eqref{eq:only:two:cutoff}. In \eqref{eq:only:two:cutoff} each term in the summation over $j$ has the form $f \, (g_\lambda \circ \Phi)$,
with 
\begin{equation}
\label{eq:fglambdaphi}
\begin{aligned}
\begin{aligned}
f & = R_{0,j}(t) \\
\Phi & = \Phi_{i_1}(t) \\
g & = \Theta_{\mu'}^j  \\
\lambda & = \lambda' \\
\end{aligned}
& \qquad \text{or} \qquad & 
\begin{aligned}
f & = R_{0,j}(t), \\
\Phi & = \Phi_{i_2}(t), \\
g & = \Theta_{\mu''}^j, \\
\lambda & = \lambda ''.
\end{aligned}
\end{aligned}
\end{equation} 
Therefore we can apply the improved H\"older inequality, Lemma \ref{p:improved:holder}, to get
\begin{equation*}
\begin{split}
\|\vartheta(t)\|_{L^1} & \leq \eta \|R_0(t)\|_{L^1} \sum_{j=1}^d \|\Theta_{\mu'}^j\|_{L^1}  
\\ 
& \qquad \qquad 
+ \frac{C(\eta, \|R_0\|_{C^1}, \max_{i=1, \dots, n} \|D \Phi_i\|_{C^0})}{\lambda'}  \sum_{j=1}^d \|\Theta_{\mu'}^j\|_{L^1} \\
& \quad + \eta \|R_0(t)\|_{L^1} \sum_{j=1}^d \|\Theta_{\mu''}^j\|_{L^1} 
\\
& \qquad \qquad 
+ \frac{C(\eta, \|R_0\|_{C^1}, \max_{i=1, \dots, n} \|D \Phi_i\|_{C^0})}{\lambda''}  \sum_{j=1}^d \|\Theta_{\mu''}^j\|_{L^1} \\
\text{(by \eqref{eq:mikado:est:1})}
& \leq \frac{M}{2} \eta \|R_0(t)\|_{L^1} 
+ C\Big(M, \eta, \|R_0\|_{C^1}, \max_{i=1, \dots, n} \|D \Phi_i\|_{C^0}\Big) \bigg( \frac{1}{\lambda'} + \frac{1}{\lambda''} \bigg). 
\end{split}
\end{equation*} 
\end{proof}

\begin{lemma}[Estimate on $\vartheta_c$]
\label{l:vartheta:c}
For every time $t \in [0,1]$,
\begin{equation*}
|\vartheta_c(t)| \leq  C\Big(M, \eta,\|R_0\|_{C^1}, \max_{i=1,\dots, N} \|D \Phi_i\|_{C^0} \Big) \bigg( \frac{1}{\lambda'} + \frac{1}{\lambda''} \bigg).
\end{equation*}
\end{lemma}
\begin{proof}
As in the proof of Lemma \ref{l:lp:vartheta}, we can use for $\vartheta(t)$ the form \eqref{eq:only:two:cutoff} and we observe that each term in the summation over $j$ has the form $f \, (g_\lambda \circ \Phi)$, with $f,\Phi,g, \lambda$ as in \eqref{eq:fglambdaphi}. We can thus apply Lemma \ref{l:mean:value} to get:
\begin{equation*}
\begin{split}
|\vartheta_c(t)| & \leq  
C\Big(\eta, \|R_0\|_{C^1}, \max_{i=1, \dots, n} \|D \Phi_i\|_{C^0}\Big) \Bigg[ \frac{1}{\lambda'}\sum_{j=1}^d \|\Theta_{\mu'}^j\|_{L^1} + \frac{1}{\lambda''} \sum_{j=1}^d \|\Theta_{\mu''}^j\|_{L^1}\Bigg] \\
\text{(by \eqref{eq:mikado:est:1})} 
& \leq C\Big(M, \eta, \|R_0\|_{C^1}, \max_{i=1, \dots, n} \|D \Phi_i\|_{C^0}\Big) \Bigg( \frac{1}{\lambda'} + \frac{1}{\lambda''} \Bigg). 
\end{split}
\end{equation*}
\end{proof}

\begin{lemma}[$C^0$ norm of $w$]
\label{l:lp:w}
For every time $t \in [0,1]$,
\begin{equation*}
\begin{split}
\|w(t)\|_{C^0(\T^d)} & \leq \frac{M}{2\eta} \max_{i=1,\dots, N} \|(D \Phi_i)^{-1}\|_{C^0 (\supp \alpha_i \times \T^d)}. 
\end{split}
\end{equation*}
\end{lemma}
\begin{proof}
As in the proof of Lemma \ref{l:lp:vartheta} we can use for $w(t)$ the form \eqref{eq:only:two:cutoff}. Therefore
\begin{equation*}
\begin{split}
\|w(t)\|_{C^0(\T^d)} 
& \leq \frac{1}{\eta} \max_{i=1,\dots, N} \|(D \Phi_i)^{-1}\|_{C^0(\supp \alpha_i \times \T^d)} \bigg( \sum_{j=1}^d \|W_{\mu'}^j\|_{L^\infty} + \|W_{\mu''}^j\|_{L^\infty} \bigg) \\
\text{(by \eqref{eq:mikado:est:1})}
& \leq \frac{M}{2\eta}  \max_{i=1,\dots, N} \|(D \Phi_i)^{-1}\|_{C^0(\supp \alpha_i \times \T^d)}.   
\end{split}
\end{equation*}
\end{proof}

\begin{lemma}[$W^{1,p}$ norm of $w$]
\label{l:w1p:w}
For every time $t \in [0,1]$, 
\begin{equation*}
\|w(t)\|_{W^{1, p}(\T^d)} \leq C \Big(M, \eta,  \max_{i=1, \dots, N} \| D \Phi_i \|_{C^1} \Big) \bigg(\lambda' (\mu')^{1 -(d-1)/p} + \lambda'' (\mu'')^{1 -(d-1)/p} \bigg).
\end{equation*}
\end{lemma}
\begin{proof}
As in the proof Lemma \ref{l:lp:w} we can use for $w(t)$ the form \eqref{eq:only:two:cutoff}. Taking one partial derivative $\partial_k$, we get
\begin{equation*}
\begin{split}
\partial_k w(t) & = \frac{\psi(t)}{\eta} \Bigg\{ \alpha_{i_1}(t) \sum_{j=1}^d \bigg[ \partial_k  (D \Phi_{i_1}(t))^{-1}  W_{\mu'}^j\big(\lambda' \Phi_{i_1}(t) \big) \\
& \qquad \qquad \qquad \qquad + \lambda' (D \Phi_{i_1}(t))^{-1} D W_{\mu'}^j \big(\lambda' \Phi_{i_1}(t) \big) D \Phi_{i_1}(t) e_k \bigg] \\
& \qquad \quad + \alpha_{i_2}(t) \sum_{j=1}^d \bigg[ \partial_k  (D \Phi_{i_2}(t))^{-1}  W_{\mu''}^j\big(\lambda'' \Phi_{i_2}(t) \big) \\
& \qquad \qquad \qquad \qquad + \lambda'' (D \Phi_{i_2}(t))^{-1} D W_{\mu''}^j \big(\lambda'' \Phi_{i_2}(t) \big) D \Phi_{i_2}(t) e_k \bigg] \Bigg\}.
\end{split}
\end{equation*}
We now apply the classical H\"older inequality to estimate $\|\partial_k w(t)\|_{L^{p}}$:
\begin{equation*}
\begin{split}
\|  \partial_k w (t)\|_{L^{p}}
& \leq \frac{\psi(t)}{\eta}  \Bigg\{ \max_{i=1, \dots, N} \big\| D (D \Phi_i)^{-1} \big\|_{C^0} \bigg( \sum_{j=1}^d \|W_{\mu'}^j\|_{L^{p}} + \|W_{\mu''}^j\|_{L^{p}} \bigg) \\
& \qquad \qquad + \max_{i=1, \dots, N} \big\|(D \Phi_i)^{-1}\big\|_{C^0} \|D\Phi_i\|_{C^0} \cdot \\
&  \qquad \qquad \qquad \qquad \qquad \qquad \cdot \bigg(\lambda' \sum_{j=1}^d  \|DW_{\mu'}^j\|_{L^{p}} + \lambda'' \sum_{j=1}^d \|DW_{\mu''}^j\|_{L^{p}} \bigg) \Bigg\} \\
& \text{(by Lemma \ref{l:diffeo})} \\
& \leq C \Big(\eta, \max_{i=1, \dots, N} \|D \Phi_i\|_{C^1} \Big) \bigg(\lambda' \sum_{j=1}^d  \|W_{\mu'}^j\|_{W^{1,p}} + \lambda'' \sum_{j=1}^d \|W_{\mu''}^j\|_{W^{1,p}} \bigg) \\
\text{(by \eqref{eq:mikado:est:2})}
& \leq 
C \Big(M, \eta,  \max_{i=1, \dots, N} \| D \Phi_i \|_{C^1} \Big) \bigg(\lambda' (\mu')^{1 -(d-1)/p} + \lambda'' (\mu'')^{1 -(d-1)/p} \bigg).
\end{split}
\end{equation*}
A similar (and even easier) computation holds for $\|w(t)\|_{L^{p}}$, thus concluding the proof of the lemma. 
\end{proof}

\section{The new defect field}
\label{s:reynolds}

In this section we continue the proof of Proposition \ref{p:main}, defining the new defect field $R_1$ and  estimating it.

\subsection{Definition of the new defect field}

We want to define $R_1$ so that
\begin{equation}
\label{eq:R1}
-\div R_1 = \partial_t \rho_1 + \div(\rho_1 u_1).
\end{equation}
Let us compute
\begin{equation}
\label{eq:new:rest:computation}
\begin{split}
\partial_t \rho_1 + \div(\rho_1 u_1) 
	& = \div (\vartheta w - R_0) \\
	& \quad + \Big[ \partial_t (\vartheta + \vartheta_c) + \div\big( (\vartheta + \vartheta_c) u_0  \big) \Big] \\
	& \quad + \div (\rho_0 w ) + \div(\vartheta_c w)  \\
	& = \div ( R^{\rm interaction}  + R^{\rm flow} + R^{\psi}+ R^{\rm quadr}) \\
	& \quad + \div R^{\rm transport} \\
	& \quad + \div R^{\rm Nash} + \div R^{\rm corr}  \\
\end{split}
\end{equation}
where we put 
\begin{equation}
\label{eq:rlin:rcorr}
\begin{split}
R^{\rm Nash} 	:=  \rho_0 w, \qquad R^{\rm corr} 	 := \vartheta_c w,
\end{split}
\end{equation}
and $R^{\rm interaction}$,  $R^{\rm flow}$, $R^{\psi}$, $R^{\rm quadr}$, $R^{\rm transport}$ will be defined respectively in \eqref{eq:r:interaction},  \eqref{eq:r:flow}, \eqref{eq:r:psi}, \eqref{eq:r:quadr}, \eqref{eq:def:rtransport} in such a way that
\begin{subequations}
\begin{align}
\div (\vartheta w - R_0) & = \div (R^{\rm interaction}  + R^{\rm flow} + R^{\psi} + R^{\rm quadr}) \label{eq:rquadr} \\
\partial_t (\vartheta + \vartheta_c) + \div\big((\vartheta + \vartheta_c) u_0 \big) & = \div R^{\rm transport}. \label{eq:rtransport}
\end{align}
\end{subequations}
We thus define
\begin{equation}
\label{eq:reynolds}
- R_1 :=  R^{\rm interaction}  +  R^{\rm flow} + R^{\psi} + R^{\rm quadr} + R^{\rm transport} + R^{\rm Nash} + R^{\rm corr},
\end{equation}
so that \eqref{eq:R1} holds.


\subsection{Definition and estimates for $R^{\rm interaction}, R^{\rm flow}, R^{\psi}, R^{\rm quadr}$}
\label{ss:r:quadr}

In this section we define and estimate the vector fields $R^{\rm quadr}$, $R^{\rm interaction}$, $R^\psi$ and $R^{\rm flow}$ so that \eqref{eq:rquadr} holds. First of all, we want to compute more explicitly $\div(\vartheta(t)w(t) - R_0(t))$, for every fixed time $t$. We can use the form \eqref{eq:only:two:cutoff} for $\vartheta(t)$ and $w(t)$. Exploiting the fact that for $j \neq k$, $\Theta_{\mu}^j$ and $W_{\mu}^k$ have disjoint support (see Proposition \ref{p:mikado}), we have
\begin{equation*}
\begin{split}
\vartheta(t) w(t) 
& = \psi^2(t) \bigg\{ \alpha_{i_1}^2(t) \sum_{j=1}^d R_{0,j}(t) (D \Phi_{i_1}(t))^{-1} \Theta_{\mu'}^j\big(\lambda' \Phi_{i_1}(t) \big) W_{\mu'}^j \big( \lambda' \Phi_{i_1}(t) \big) \\
& \qquad \qquad \quad +  \alpha_{i_2}^2(t) \sum_{j=1}^d R_{0,j}(t) (D \Phi_{i_2}(t))^{-1} \Theta_{\mu''}^j\big(\lambda'' \Phi_{i_2}(t) \big) W_{\mu''}^j \big( \lambda'' \Phi_{i_2}(t) \big) \bigg\} \\
& \quad + R^{\rm interaction}(t),
\end{split}
\end{equation*}
where we set
\begin{equation}
\label{eq:r:interaction}
\begin{split}
R^{\rm i}& ^{\rm nteraction}(t) \\
& := \psi^2(t) \alpha_{i_1}(t) \alpha_{i_2}(t) \sum_{j,k=1}^d\bigg[ R_{0,j}(t) (D\Phi_{i_2}(t))^{-1} \Theta_{\mu'}^j \big(\lambda' \Phi_{i_1}(t) \big) W_{\mu''}^k \big(\lambda'' \Phi_{i_2}(t) \big)   \\
& \qquad \qquad \qquad \quad \qquad + R_{0,j}(t) (D \Phi_{i_1}(t))^{-1} \Theta_{\mu''}^j\big(\lambda'' \Phi_{i_2}(t) \big) W_{\mu'}^k\big(\lambda' \Phi_{i_1}(t) \big)
\bigg].
\end{split}
\end{equation}
On the other side, using the fact that $\sum_{i=1}^N \alpha_i^2 \equiv 1$, we can write
\begin{equation*}
\begin{split}
R_0(t) & = \psi^2(t) R_0(t) + R_0(t) \big[ 1 - \psi^2(t) \big] \\
& = \psi^2(t) R_0(t) - R^{\psi}(t) \\
& \text{(using that {$\textstyle \sum\nolimits_{i=1}^N \alpha_i^2 \equiv 1$})} \\
& = \psi^2(t) \Big\{  \alpha_{i_1}^2(t) R_0(t) + \alpha_{i_2}^2(t) R_0(t) \Big\} - R^{\psi}(t) \\
& = \psi^2(t) \Big\{ \alpha_{i_1}^2(t) \big[D\Phi_{i_1}(t)\big]^{-1} R_0(t) +  \alpha_{i_2}^2(t) \big[D \Phi_{i_2}(t)\big]^{-1} R_0(t) \Big\} 
- R^{\rm flow}(t) - R^{\psi}(t) \\
& = \psi^2(t) \bigg\{ \alpha_{i_1}^2(t) \sum_{j=1}^d R_{0,j}(t) \big[D\Phi_{i_1}(t)\big]^{-1} e_j +  \alpha_{i_2}^2(t) \sum_{j=1}^d R_{0,j}(t) \big[D \Phi_{i_2}(t)\big]^{-1} e_j \bigg\} \\
& \qquad - R^{\rm flow}(t) - R^{\psi}(t),
\end{split}
\end{equation*}
where we set
\begin{equation}
\label{eq:r:flow}
\begin{split}
- R  ^{\rm flow}(t) 
& := \psi^2(t) \bigg\{ \alpha_{i_1}^2(t) \Big[ {\rm Id} -  (D\Phi_{i_1}(t))^{-1} \Big] R_0(t)   
\\
& \qquad \qquad \quad
+ \alpha_{i_2}^2(t) \Big[ {\rm Id} -  (D\Phi_{i_2}(t))^{-1} \Big] R_0(t) \bigg\}
\end{split}
\end{equation}
with ${\rm Id}$ being the identity matrix, and
\begin{equation}
\label{eq:r:psi}
- R^{\psi}(t)  := R_0(t) \big[ 1 - \psi^2(t) \big].
\end{equation}
Summarizing, we have
\begin{equation*}
\begin{split}
\div & \big( \vartheta(t) w(t) - R_0(t) \big) \\
& = \div R^{\rm interaction}(t) + \div R^{\rm flow}(t) + \div R^{\psi}(t) \\
& \quad + \psi^2(t) \Bigg\{ \alpha_{i_1}^2(t) \sum_{j=1}^d \div \bigg[ R_{0,j} \big[ D \Phi_{i_1}(t) \big]^{-1} \Big( \Theta_{\mu'}^j \big(\lambda' \Phi_{i_1}(t) \big) W_{\mu'}^j \big( \lambda' \Phi_{i_1}(t) \big) - e_j \Big) \bigg] \\
& \qquad \qquad \quad + \alpha_{i_2}^2(t) \sum_{j=1}^d \div \bigg[ R_{0,j} \big[ D \Phi_{i_2}(t) \big]^{-1} \Big( \Theta_{\mu''}^j \big(\lambda'' \Phi_{i_2}(t) \big) W_{\mu''}^j \big( \lambda'' \Phi_{i_2}(t) \big) - e_j \Big) \bigg] \Bigg\} \\
& = \div R^{\rm interaction}(t) + \div R^{\rm flow}(t) + \div R^{\psi}(t) \\
& \quad + \psi^2(t) \Bigg\{ \alpha_{i_1}^2(t) \sum_{j=1}^d   \nabla R_{0,j} \cdot \big[ D \Phi_{i_1}(t) \big]^{-1} \Big( \Theta_{\mu'}^j \big(\lambda' \Phi_{i_1}(t) \big) W_{\mu'}^j \big( \lambda' \Phi_{i_1}(t) \big) - e_j \Big)  \\
& \qquad \qquad \quad + \alpha_{i_2}^2(t) \sum_{j=1}^d   \nabla R_{0,j} \cdot \big[ D \Phi_{i_2}(t) \big]^{-1} \Big( \Theta_{\mu''}^j \big(\lambda'' \Phi_{i_2}(t) \big) W_{\mu''}^j \big( \lambda'' \Phi_{i_2}(t) \big) - e_j \Big)  \Bigg\}, \\
\end{split}
\end{equation*}
where in the last equality we used the fact that $\div ((\Theta_\mu^j)_\lambda (W_\mu^j)_\lambda - e_j) = 0$ for every $\mu, \lambda, j$ (see Proposition \ref{p:mikado}) and Lemma \ref{l:divergence:diffeo}. 
We now observe that each term in the two summations over $j$ has zero mean value (being a divergence) and it has the form $f (D \Phi)^{-1} (g_\lambda \circ \Phi)$, for
\begin{equation*}
\begin{aligned}
\begin{aligned}
f & = \nabla R_{0,j}(t) \\
\Phi & = \Phi_{i_1}(t) \\
g & = \Theta_{\mu'}^j W_{\mu'}^j - e_j \\
\lambda & = \lambda' 
\end{aligned}
& \quad \text{ or } \quad & 
\begin{aligned}
f & = \nabla R_{0,j}(t) \\
\Phi & = \Phi_{i_2}(t) \\
g & = \Theta_{\mu''}^j W_{\mu''}^j - e_j \\
\lambda & = \lambda'' 
\end{aligned}
\end{aligned}
\end{equation*} 
We can therefore apply Lemma \ref{p:antidiv:transport} and define
\begin{equation}
\label{eq:r:quadr}
\begin{split}
&R ^{\rm quadr}(t) \\ 
& := \psi^2(t)  \Bigg\{ \alpha_{i_1}^2(t) \sum_{j=1}^d \mathcal{R} \bigg(\nabla R_{0,j} \cdot  \big[ D \Phi_{i_1}(t) \big]^{-1} \Big( \Theta_{\mu'}^j \big(\lambda' \Phi_{i_1}(t) \big) W_{\mu'}^j \big( \lambda' \Phi_{i_1}(t) \big) - e_j \Big) \bigg) \\
& \quad \ \ + \alpha_{i_2}^2(t) \sum_{j=1}^d \mathcal{R} \bigg( \nabla R_{0,j} \cdot \big[ D \Phi_{i_2}(t) \big]^{-1} \Big( \Theta_{\mu''}^j \big(\lambda'' \Phi_{i_2}(t) \big) W_{\mu''}^j \big( \lambda'' \Phi_{i_2}(t) \big) - e_j \Big) \bigg) \Bigg\},
\end{split}
\end{equation}
so that \eqref{eq:rquadr} holds. We now separately estimate $R^{\rm interaction}$, $R^{\rm flow}$, $R^{\psi}$, $R^{\rm quadr}$. We start with $R^{\rm interaction}$. 

\begin{lemma}
\label{l:rinteraction}
For every time $t$ it holds
\begin{equation*}
\begin{split}
\|R^{\rm interaction} & (t)\|_{L^1(\T^d)} \\
& \leq C \Big(M,  \|R_0\|_{C^0}, \max_{i=1,\dots, N} \|D \Phi_i\|_{C^0} \Big)  \bigg( \frac{1}{(\mu')^{d-1}} + \frac{1}{(\mu'')^{d-1}} + \frac{\lambda' \mu'}{\lambda''} + \frac{\lambda' (\mu')^d}{\lambda'' (\mu'')^{d-1}} \bigg).
\end{split}
\end{equation*}
\end{lemma}
\begin{proof}
Consider the definition \eqref{eq:r:interaction} of $R^{\rm interaction}$. We start by estimating $\|\Theta_{\mu'}^j (\lambda' \Phi_{i_1}(t))\|_{C^1}$ and $\|W_{\mu'}^k(\lambda' \Phi_{i_1}(t))\|_{C^1}$, using \eqref{eq:mikado:est:3} and the chain rule
\begin{equation}
\label{eq:bad:term}
\begin{split}
\|\Theta_{\mu'}^j (\lambda' \Phi_{i_1}(t))\|_{C^1} & \leq C(M, \max_{i=1,\dots, N} \|D \Phi_i\|_{C^0} ) \lambda' (\mu')^d, \\
\|W_{\mu'}^k (\lambda' \Phi_{i_1}(t))\|_{C^1} & \leq C(M, \max_{i=1,\dots, N} \|D \Phi_i\|_{C^0} ) \lambda' \mu'. \\
\end{split}
\end{equation}
We now estimate  $\Theta_{\mu'}^j (\lambda' \Phi_{i_1}(t) ) W_{\mu''}^k (\lambda'' \Phi_{i_2}(t) )$, using the improved H\"older inequality, Lemma \ref{p:improved:holder} and  considering $\lambda''$ as the \emph{fast} oscillation. We have
\begin{equation*}
\begin{split}
\big\| \Theta_{\mu'}^j \big(\lambda' & \Phi_{i_1}(t) \big) W_{\mu''}^k \big(\lambda'' \Phi_{i_2}(t) \big) \big\|_{L^1} \\
& \leq \| \Theta_{\mu'}^j \|_{L^1} \| W_{\mu''}^k  \|_{L^1} + \frac{1}{\lambda''} \big\| \Theta_{\mu'}^j \big(\lambda' \Phi_{i_1}(t) \big)\big\|_{C^1}  \big\|D \Phi_{i_2}(t)\big\|^{d-1}_{C^0} \big\| W_{\mu''}^k  \big\|_{L^1} \\
& \leq  C\Big(M, \max_{i=1,\dots, N} \|D \Phi_i\|_{C^0} \Big)  \bigg(\frac{1}{(\mu'')^{d-1}} + \frac{\lambda' (\mu')^d}{\lambda'' (\mu'')^{d-1}} \bigg),
\end{split}
\end{equation*} 
where in the last inequality we used \eqref{eq:mikado:est:1}, \eqref{eq:mikado:est:2} and \eqref{eq:bad:term}. 
A similar estimate holds for $\Theta_{\mu''}^j (\lambda'' \Phi_{i_2}(t)) W_{\mu'}^k (\lambda' \Phi_{i_1}(t))$:
\begin{equation*}
\begin{split}
\big\| \Theta_{\mu''}^j (\lambda'' & \Phi_{i_2}(t)) W_{\mu'}^k (\lambda' \Phi_{i_1}(t)) \big\|_{L^1} \\
& \leq \| \Theta_{\mu''}^j \|_{L^1} \|  W_{\mu'}^k \|_{L^1} + \frac{1}{\lambda''} \|  W_{\mu'}^k (\lambda' \Phi_{i_1}(t)) \|_{C^1} \|D \Phi_{i_2}(t)\|^{d-1}_{C^0} \| \Theta_{\mu''}^j \|_{L^1} \\
& \leq  C\Big(M, \max_{i=1,\dots, N} \|D \Phi_i\|_{C^0} \Big) \bigg( \frac{1}{(\mu')^{d-1}} + \frac{\lambda' \mu'}{\lambda''} \bigg).
\end{split}
\end{equation*}
Therefore
\begin{equation*}
\begin{split}
\|& R^{\rm interaction}(t)\|_{L^1} \\
& \leq C \Big( M, \|R_0\|_{C^0}, \max_{i=1,\dots, N} \|D \Phi_i\|_{C^0} \Big) \cdot \\
& \qquad \cdot \sum_{j,k=1}^d \bigg[   
\big\| \Theta_{\mu'}^j \big(\lambda'  \Phi_{i_1}(t) \big) W_{\mu''}^k \big(\lambda'' \Phi_{i_2}(t) \big) \big\|_{L^1}
+ \big\| \Theta_{\mu''}^j (\lambda''  \Phi_{i_2}(t)) W_{\mu'}^k (\lambda' \Phi_{i_1}(t)) \big\|_{L^1} \bigg] \\
& \leq C \big(M, \|R_0\|_{C^0}, \max_{i=1,\dots, N} \|D \Phi_i\|_{C^0} \big)  \bigg( \frac{1}{(\mu')^{d-1}} + \frac{1}{(\mu'')^{d-1}} + \frac{\lambda' \mu'}{\lambda''} + \frac{\lambda' (\mu')^d}{\lambda'' (\mu'')^{d-1}} \bigg). 
\end{split}
\end{equation*}
\end{proof}

\begin{lemma}
\label{l:rflow}
For every $t \in [0,1]$, 
\begin{equation*}
\|R^{\rm flow}(t)\|_{L^1(\T^d)} \leq \|R_0\|_{C^0} \max_{i=1,\dots, N} \big\| {\rm Id} - D\Phi_i(t)^{-1} \big\|_{C^0(\supp \alpha_i \times \T^d)}. 
\end{equation*}
\end{lemma}
\begin{proof}
The proof follows immediately from the definition of $R^{\rm flow}$.
\end{proof}

\begin{lemma}
\label{l:rpsi}
For every $t \in [0,1]$,
\begin{equation*}
\|R^\psi(t) \|_{L^1(\T^d)} \leq \delta/4.
\end{equation*}
\end{lemma}
\begin{proof}
If $\psi^2(t) \neq 1$, then, by \eqref{eq:psi}, $\|R_0(t)\|_{L^1} \leq \delta/4$ and thus the conclusion follows.
\end{proof}

\begin{lemma}
\label{l:rquadr}
For every $t \in [0,1]$, 
\begin{equation*}
\|R^{\rm quadr}(t)\|_{L^1(\T^d)} \leq C \Big(M, \|R_0\|_{C^2}, \max_{i=1,\dots, N} \|D \Phi_i\|_{C^1} \Big) \bigg( \frac{1}{\lambda'} + \frac{1}{\lambda''} \bigg).
\end{equation*}
\end{lemma}
\begin{proof}
$R^{\rm quadr}(t)$ is defined in \eqref{eq:r:quadr} using Lemma \ref{p:antidiv:transport}. Applying the bounds provided by such proposition, with $k=0$ and $p=1$, we get
\begin{equation*}
\begin{split}
\|& R^{\rm quadr}(t)\|_{L^1} \\
& \leq C \Big( \|R_0\|_{C^2}, \max_{i=1,\dots, N} \|D \Phi_i\|_{C^1} \Big) \bigg( \frac{1}{\lambda'} \sum_{j=1}^d \|\Theta_{\mu'}^j W_{\mu'}^j - e_j\|_{L^1} + \frac{1}{\lambda''} \sum_{j=1}^d \|\Theta_{\mu''}^j W_{\mu''}^j - e_j\|_{L^1}\bigg) \\
& \text{(by \eqref{eq:mikado:est:1})} \\
& \leq C \Big(M, \|R_0\|_{C^2}, \max_{i=1,\dots, N} \|D \Phi_i\|_{C^1} \Big) \bigg( \frac{1}{\lambda'} + \frac{1}{\lambda''} \bigg).  
\end{split} 
\end{equation*}
\end{proof}

\subsection{Definition and estimates for $R^{\rm transport}$}
\label{ss:r:transport}

In this section we define and estimate the vector fields $R^{\rm transport}$ so that \eqref{eq:rtransport} holds. First of all, we want to compute more explicitly $\partial_t (\vartheta(t) + \vartheta_c(t)) + \div((\vartheta(t)+\vartheta_c(t)) u_0(t))$, for every fixed time $t$. We can use the local form \eqref{eq:only:two:cutoff} for $\vartheta(t)$ and $w(t)$. 
We have
\begin{equation}
\label{eq:transport:explicit}
\begin{split}
\partial_t &  (\vartheta + \vartheta_c) + \div ((\vartheta + \vartheta_c)u_0) \\
& =  \vartheta_c'(t) +  \sum_{j=1}^d \bigg\{ A^j_1(t,x) \Theta_{\mu'}^j \big(\lambda' \Phi_{i_1} \big) +  A^j_2(t,x)  \Theta_{\mu''}^j \big(\lambda'' \Phi_{i_2} \big) \\
& \qquad \qquad \qquad \qquad + \lambda' B^j_1(t,x) \cdot \Big[ \partial_t \Phi_{i_1} + (u_0 \cdot \nabla) \Phi_{i_1} \Big] \\
& \qquad \qquad \qquad \qquad + \lambda'' B^j_2(t,x) \cdot \Big[ \partial_t \Phi_{i_2} + (u_0 \cdot \nabla) \Phi_{i_2} \Big] \bigg\} \\
& = \vartheta_c'(t) +  \sum_{j=1}^d \bigg\{ A^j_1(t,x) \Theta_{\mu'}^j \big(\lambda' \Phi_{i_1} \big) + A^j_2(t,x)  \Theta_{\mu''}^j \big(\lambda'' \Phi_{i_2} \big) \bigg\},
\end{split}
\end{equation}
where
\begin{equation*}
\begin{split}
A_1^j & := \eta \, \bigg[\psi' \alpha_{i_1} R_{0,j} + \psi \alpha_{i_1}' R_{0,j} + \psi\alpha_{i_1} \Big( \partial_t R_{0,j} + \nabla R_{0,j} \cdot u_0 \Big) \bigg], \\
A_2^j & := \eta \,  \bigg[\psi' \alpha_{i_2} R_{0,j} + \psi \alpha_{i_2}' R_{0,j} + \psi \alpha_{i_2} \Big( \partial_t R_{0,j} + \nabla R_{0,j} \cdot u_0 \Big) \bigg], \\
B_1^j & := \eta \, \psi  \alpha_{i_1} R_{0,j} \, \nabla \Theta_{\mu'}^j(\lambda' \Phi_{i_1}), \\
B_2^j & := \eta \, \psi \alpha_{i_2} R_{0,j} \, \nabla \Theta_{\mu''}^j(\lambda'' \Phi_{i_2})
\end{split}
\end{equation*}
and we used \eqref{eq:inverse:flow}.  We now continue the chain of equalities in \eqref{eq:transport:explicit}, by adding and subtracting the mean value of each term in the summations over $j$, as follows:
\begin{equation}
\label{eq:transport:explicit:2}
\begin{split}
\partial_t &  (\vartheta + \vartheta_c) + \div ((\vartheta + \vartheta_c)u_0) \\
& = \sum_{j=1}^d \Bigg\{\bigg(  A^j_1(t,x) \Theta_{\mu'}^j \big(\lambda' \Phi_{i_1} \big) - \fint_{T^d} A^j_1(t,x) \Theta_{\mu'}^j \big(\lambda' \Phi_{i_1} \big) dx \bigg) \\ 
& \qquad \qquad \qquad + \bigg( A^j_2(t,x)  \Theta_{\mu''}^j \big(\lambda'' \Phi_{i_2} \big) - \fint_{T^d} A^j_2(t,x)  \Theta_{\mu''}^j \big(\lambda'' \Phi_{i_2} \big) dx \bigg) \Bigg\}  \\
& \quad + \vartheta_c'(t) + \sum_{j=1}^d \Bigg\{ \fint_{T^d} A^j_1(t,x) \Theta_{\mu'}^j \big(\lambda' \Phi_{i_1} \big) dx + \fint_{T^d} A^j_2(t,x)  \Theta_{\mu''}^j \big(\lambda'' \Phi_{i_2} \big) dx \Bigg\} \\
& =  \sum_{j=1}^d \Bigg\{\bigg(  A^j_1(t,x) \Theta_{\mu'}^j \big(\lambda' \Phi_{i_1} \big) - \fint_{T^d} A^j_1(t,x) \Theta_{\mu'}^j \big(\lambda' \Phi_{i_1} \big) dx \bigg) \\ 
& \qquad \qquad \qquad + \bigg( A^j_2(t,x)  \Theta_{\mu''}^j \big(\lambda'' \Phi_{i_2} \big) - \fint_{T^d} A^j_2(t,x)  \Theta_{\mu''}^j \big(\lambda'' \Phi_{i_2} \big) dx \bigg) \Bigg\}.
\end{split}
\end{equation}
The last equality is a consequence of the fact that 
\begin{equation*}
\fint \Big(\partial_t (\vartheta + \vartheta_c) + \div ((\vartheta + \vartheta_c)u_0)  \Big) dx = 0.
\end{equation*}
We now observe that each term in the last line in \eqref{eq:transport:explicit:2} has the form $f \, (g_\lambda \circ \Phi) - \fint f \, (g_\lambda \circ \Phi) dx$ for
\begin{equation*}
\begin{aligned}
f & = A_1^j & \text{ or } && f & = A_2^j, \\
\Phi & = \Phi_{i_1}(t) & \text{ or } && \Phi & = \Phi_{i_2}(t), \\
g & = \Theta_{\mu'}^j & \text{ or } && g & = \Theta_{\mu''}^j, \\
\lambda & = \lambda' & \text{ or } && \lambda & = \lambda ''.
\end{aligned}
\end{equation*} 
Since $\Theta_\mu^j$ has zero mean value (see Proposition \ref{p:mikado}), we can apply Lemma \ref{p:antidiv:transport} and define
\begin{equation}
\label{eq:def:rtransport}
\begin{split}
& R^{\rm transport}(t) \\
& := \sum_{j=1}^d \Bigg\{ \mathcal{R} \bigg(  A^j_1(t,x) \Theta_{\mu'}^j \big(\lambda' \Phi_{i_1} \big) - \fint_{T^d} A^j_1(t,x) \Theta_{\mu'}^j \big(\lambda' \Phi_{i_1} \big) dx \bigg) \\
& \qquad \qquad \qquad + \mathcal{R} \bigg( A^j_2(t,x)  \Theta_{\mu''}^j \big(\lambda'' \Phi_{i_2} \big) - \fint_{T^d} A^j_2(t,x)  \Theta_{\mu''}^j \big(\lambda'' \Phi_{i_2} \big) dx \bigg) \Bigg\}.
\end{split}
\end{equation}

\begin{lemma}
\label{l:r:transport}
For every $t \in [0,1]$, it holds
\begin{equation*}
\|R^{\rm transport}(t)\|_{L^1(\T^d)} \leq C \big(M, \eta, \delta, \tau, \|R_0\|_{C^2}, \|u_0\|_{C^1}, \max_{i=1,\dots, N} \|D \Phi_i\|_{C^1} \big) \bigg( \frac{1}{\lambda'} + \frac{1}{\lambda''} \bigg). 
\end{equation*}
\end{lemma}
\begin{proof}
First of all, we observe that
\begin{equation*}
|\psi'(t)| \leq C\big(\delta, \|R_0\|_{C^1} \big), \qquad |\alpha_i'(t)| \leq C(\tau) \text{ for all } i=1,\dots, N.
\end{equation*}
Therefore
\begin{equation*}
\|A_1^j(t)\|_{C^1(\T^d)}, \|A_2^j(t)\|_{C^1(\T^d)} \leq C \big( \eta, \delta, \tau, \|R_0\|_{C^2}, \|u_0\|_{C^1}  \big).
\end{equation*}
We defined $R^{\rm transport}$ in \eqref{eq:def:rtransport} using the antidivergence operator provided by Lemma \ref{p:antidiv:transport}. We can thus apply the bounds provided by such proposition, with $k=0$ and $p=1$, to get
\begin{equation*}
\begin{split}
\|& R^{\rm transport}(t)\|_{L^1} \\
& \leq \sum_{j=1}^d \Bigg\{ \frac{\|A_1^j(t)\|_{C^1(\T^d)} \|D\Phi_{i_1}\|_{C^1}^{d-1} \|\Theta_{\mu'}^j\|_{L^1}}{\lambda'} + \frac{\|A_2^j(t)\|_{C^1(\T^d)} \|D\Phi_{i_2}\|_{C^1}^{d-1} \|\Theta_{\mu''}^j\|_{L^1}}{\lambda''} \Bigg\} \\
& \leq C \big(M, \eta, \delta, \tau, \|R_0\|_{C^2}, \|u_0\|_{C^1}, \max_{i=1,\dots, N} \|D \Phi_i\|_{C^1} \big) \bigg( \frac{1}{\lambda'} + \frac{1}{\lambda''} \bigg),
\end{split}
\end{equation*}
where in the last line we used \eqref{eq:mikado:est:1}. 
\end{proof}

\subsection{Estimates for $R^{\rm Nash}$ and $R^{\rm corr}$}

In this section we estimate $R^{\rm Nash}$ and $R^{\rm corr}$. 

\begin{lemma}
\label{l:r:nash}
For every $t \in [0,1]$, 
\begin{equation*}
\|R^{\rm Nash}(t)\|_{L^1(\T^d)} \leq C\big(\|\rho_0(t)\|_{C^0}, \max_{i=1,\dots, N} \|D \Phi_i\|_{C^0} \big) 
 \bigg( \frac{1}{(\mu')^{d-1}} + \frac{1}{(\mu'')^{d-1}} \bigg). 
\end{equation*}
\end{lemma}
\begin{proof}
We have
\begin{equation*}
\begin{split}
\|R^{\rm Nash}(t)\|_{L^1}
& = \|\rho_0(t) w(t)\|_{L^1} \\
& \leq \|\rho_0(t)\|_{C^0} \|w(t)\|_{L^1} \\
& \leq C\big(\eta, \|\rho_0(t)\|_{C^0}, \max_{i=1,\dots, N} \|D \Phi_i\|_{C^0} \big) \sum_{j=1}^d \bigg( \|W_{\mu'}^j\|_{L^1} + \|W_{\mu''}^j\|_{L^1} \bigg) \\
\text{(by \eqref{eq:mikado:est:2})}
& \leq C\big(\eta, \|\rho_0(t)\|_{C^0}, \max_{i=1,\dots, N} \|D \Phi_i\|_{C^0} \big) \bigg( \frac{1}{(\mu')^{d-1}} + \frac{1}{(\mu'')^{d-1}} \bigg).  
\end{split}
\end{equation*}
\end{proof}

\begin{lemma}
\label{l:r:corr}
For every $t \in [0,1]$, 
\begin{equation*}
\|R^{\rm corr}(t)\|_{L^1(\T^d)} \leq  C(M, \eta, \|R_0\|_{C^1}, \max_{i=1, \dots, N} \|D \Phi_i\|_{C^0} ) \bigg(\frac{1}{\lambda'} + \frac{1}{\lambda''} \bigg). 
\end{equation*}
\end{lemma}
\begin{proof}
%
%
%

We use Lemma \ref{l:vartheta:c} and Lemma \ref{l:lp:w}:
\begin{equation*}
\begin{split}
\|\vartheta_c(t) w(t)\|_{L^1}
& = |\vartheta_c(t)| \|w(t)\|_{L^1}  \\
& \leq |\vartheta_c(t)| \|w(t)\|_{C^0}  \\
& \leq C(M, \eta, \|R_0\|_{C^1}, \max_{i=1, \dots, N} \|D \Phi_i\|_{C^0} ) \bigg(\frac{1}{\lambda'} + \frac{1}{\lambda''} \bigg). 
\end{split}
\end{equation*}
%
\end{proof}

\section{Proof of Proposition \ref{p:main}}
\label{s:proof:prop}

In this section we conclude the proof of Proposition \ref{p:main}, and thus also the proof of Theorem \ref{thm:main} and, consequently, the proof of Theorem \ref{thm:nonuniq}. We first prove that if $R_0(t) = 0$ at some time $t \in [0,1]$, then $R_1(t) = 0$. Observe that if $R_0(t) = 0$, then by Remark \ref{rmk:zero:if:zero}, 
\begin{equation*}
\rho_1(t) - \rho_0(t) = \vartheta(t) + \vartheta_c(t) = 0, \qquad u_1(t) - u_0(t) = w(t) = 0.
\end{equation*}
Moreover, by \eqref{eq:psi}, $\psi \equiv 0$ on a neighborhood of $t$ and thus $\psi(t) = \psi'(t) = 0$. 
Therefore
\begin{equation*}
\begin{aligned}
\psi(t) = 0 \quad & \Longrightarrow & R^{\rm interaction}(t) = R^{\rm flow}(t) = R^{\rm quadr}(t)  = 0, \\
\psi(t) = \psi'(t) = 0 \quad & \Longrightarrow & R^{\rm transport}(t) = 0, \\
R_0(t) = 0 \quad & \Longrightarrow & R^\psi(t) = 0, \\
w(t) = w_c(t) = 0 \quad & \Longrightarrow & R^{\rm Nash}(t) = R^{\rm corr}(t) = 0,
\end{aligned}
\end{equation*}
and thus $R_1(t) = 0$. 

We now prove estimates \eqref{eq:dist:rho:stat}-\eqref{eq:reyn:stat}. First of all, in view of Lemma \ref{l:lp:w} and Lemma \ref{l:rflow}, we choose $\tau$ so small that
\begin{subequations}
\begin{align}
\max_{i=1,\dots, N} \|D \Phi_i^{-1}\|_{C^0 (\supp \alpha_i \times \T^d)} & \leq 2, \label{eq:dphi} \\
\|R_0\|_{C^0} \max_{i=1,\dots, N} \|{\rm Id} - D \Phi_i^{-1}\|_{C^0 (\supp \alpha_i \times \T^d)} & \leq \frac{\delta}{4}. \label{eq:dphi:id}
\end{align}
\end{subequations}
This is always possible since, by \eqref{eq:inverse:flow}, $\Phi_i(t_i, x) = x$ and thus $D \Phi_i(t_i, x) = {\rm Id}$ for every $i=1,\dots, N$. We choose also $\lambda', \mu', \lambda'', \mu''$ such that
$1 \ll \lambda' \ll \mu' \ll \lambda'' \ll \mu''$. More precisely, we set
\begin{equation*}
\lambda' = \lambda, \qquad \mu' := \lambda^\alpha, \qquad \lambda'' := \lambda^\beta, \qquad \mu'' := \lambda^\gamma,
\end{equation*}
for some
\begin{equation*}
1 < \alpha < \beta < \gamma
\end{equation*}
and $\lambda \gg 1$ to be fixed later. 

\bigskip
\noindent \textit{1. Estimate \eqref{eq:dist:rho:stat}.} If $R_0(t) = 0$, we have already seen that $\rho_1(t) = \rho_0(t)$. We can thus assume $R_0(t) \neq 0$. We have
\begin{equation*}
\begin{split}
\|\rho_1(t) - \rho_0(t)\|_{L^1}
& \leq \|\vartheta_0(t)\|_{L^1} + |\vartheta_c(t)| \\
& \text{(by Lemmas \ref{l:lp:vartheta} and \ref{l:vartheta:c})}  \\
& \leq \frac{M\eta}{2} \|R_0(t)\|_{L^1} + C \Big(M, \eta, \|R_0\|_{C^1}, \max_{i=1,\dots, N} \|D \Phi_i\|_{C^0} \Big) \bigg( \frac{1}{\lambda'} + \frac{1}{\lambda''} \bigg) \\
& \leq \frac{M\eta}{2} \|R_0(t)\|_{L^1} + C \Big(M, \eta, \|R_0\|_{C^1}, \max_{i=1,\dots, N} \|D \Phi_i\|_{C^0} \Big) \bigg( \frac{1}{\lambda} + \frac{1}{\lambda^\beta} \bigg) \\
& \leq M \eta \|R_0(t)\|_{L^1} 
\end{split}
\end{equation*}
if the constant $\lambda$ is chosen large enough.

\bigskip
\noindent \textit{2. Estimate \eqref{eq:dist:u:1:stat}.} We have
\begin{equation*}
\begin{split}
\|u_1(t) - u_0(t)\|_{C^0} 
& \leq \|w(t)\|_{C^0} \\
\text{(by Lemma \ref{l:lp:w})} 
& \leq \frac{M}{2\eta} \max_{i=1,\dots, N} \|(D \Phi_i)^{-1}\|_{C^0 (\supp \alpha_i \times \T^d)} \\
\text{(by \eqref{eq:dphi})} 
& \leq \frac{M}{\eta}.
\end{split}
\end{equation*}

\bigskip
\noindent \textit{3. Estimate \eqref{eq:dist:u:2:stat}.} We have
\begin{equation*}
\begin{split}
\|u_1(t) - u_0(t)\|_{W^{1, p}} 
& \leq \|w(t)\|_{W^{1, p}}\\
& \text{(by Lemma \ref{l:w1p:w})} \\
& \leq C \Big(M, \eta,  \max_{i=1, \dots, N} \| D \Phi_i \|_{C^1} \Big) \bigg(\lambda' (\mu')^{1 -(d-1)/p} + \lambda'' (\mu'')^{1 -(d-1)/p} 
\bigg) \\
& \leq C \Big(M, \eta,  \max_{i=1, \dots, N} \| D \Phi_i \|_{C^1} \Big) \bigg( \lambda^{1 + \alpha(1 - (d-1)/p)} + \lambda^{\beta + \gamma(1 - (d-1)/p)}
\bigg) \\
& \leq \delta,
\end{split}
\end{equation*}
if $\alpha$, $\beta, \gamma$ are chosen so that
\begin{subequations}
\label{eq:cond:1}
\begin{align}
1 + \alpha \bigg( 1 - \frac{d-1}{p} \bigg) & < 0, \label{eq:cond:alpha} \\
\beta + \gamma \bigg( 1 - \frac{d-1}{p} \bigg) & < 0 \label{eq:cond:gamma:1}
\end{align}
\end{subequations}
and $\lambda$ is large enough. 

\bigskip
\noindent \textit{4. Estimate \eqref{eq:reyn:stat}.} 
Recall the definition of $R_1$ in \eqref{eq:reynolds}.  
Using Lemmas \ref{l:rinteraction}, \ref{l:rflow}, \ref{l:rpsi}, \ref{l:rquadr}, \ref{l:r:transport}, \ref{l:r:nash}, \ref{l:r:corr} and \eqref{eq:dphi:id}, we get
\begin{equation*}
\begin{split}
\|R_1(t)\|_{L^1} 
& \leq \|R^{\rm interaction}(t)\|_{L^1} + \|R^{\rm flow}(t)\|_{L^1} + \|R^\psi(t)\|_{L^1} + \|R^{\rm quadr}(t)\|_{L^1} \\
& \quad + \|R^{\rm transport}(t)\|_{L^1} + \|R^{\rm Nash}(t)\|_{L^1} + \|R^{\rm corr}(t)\|_{L^1} \\
& \leq \frac{\delta}{2} + C \Big(M, \eta, \delta, \tau, \|\rho_0\|_{C^0}, \|u_0\|_{C^1},  \|R_0\|_{C^2}, \max_{i=1,\dots, N} \|D \Phi_i\|_{C^1} \Big) \cdot \\
& \qquad \qquad \cdot \Bigg[ \frac{1}{\lambda'} + \frac{1}{\lambda''} + \frac{1}{(\mu')^{d-1}} + \frac{1}{(\mu'')^{d-1}} + \frac{\lambda' \mu'}{\lambda''} + \frac{\lambda' (\mu')^d}{\lambda'' (\mu'')^{d-1}} \Bigg] \\
& \leq \frac{\delta}{2} + C \Big(M, \eta, \delta, \tau, \|\rho_0\|_{C^0}, \|u_0\|_{C^1},  \|R_0\|_{C^2}, \max_{i=1,\dots, N} \|D \Phi_i\|_{C^1} \Big) \cdot \\
& \qquad \qquad \cdot \Bigg[ \frac{1}{\lambda} + \frac{1}{\lambda^\beta} + \frac{1}{\lambda^{\alpha(d-1)}} + \frac{1}{\lambda^{\gamma(d-1)}} + \frac{\lambda^{1 + \alpha}}{\lambda^\beta} + \frac{\lambda^{1 + \alpha d}}{\lambda^{\beta + \gamma(d-1)}} \Bigg] \\ 
& \leq \frac{\delta}{2} + \frac{\delta}{2} \\
& \leq \delta,
\end{split}
\end{equation*}
if
\begin{subequations}
\label{eq:cond:2}
\begin{align}
1 + \alpha - \beta & < 0, \label{eq:cond:beta} \\
1 + \alpha d - \beta - \gamma(d-1) & < 0. \label{eq:cond:gamma:2}
\end{align}
\end{subequations}
and $\lambda$ is large enough. 

\bigskip
We still have to choose $\alpha, \beta, \gamma$ so that  \eqref{eq:cond:1}, \eqref{eq:cond:2} are satisfied. This can be easily done as follows, recalling that $p < d-1$. First we fix $\alpha > 1$ so that
\begin{equation*}
\alpha > \frac{1}{\frac{d-1}{p} - 1}
\end{equation*} 
so that \eqref{eq:cond:alpha} is satisfied. Then we choose $\beta$ so that
\begin{equation*}
\beta > 1 + \alpha,
\end{equation*}
so that \eqref{eq:cond:beta} is satisfied. Finally we choose $\gamma>1$ so that
\begin{equation*}
\gamma > \frac{\beta}{\frac{d-1}{p} - 1} \text{ and } \gamma > \frac{1 + \alpha d - \beta}{d-1}
\end{equation*}
so that \eqref{eq:cond:gamma:1} and \eqref{eq:cond:gamma:2} are satisfied. This concludes the proof of Proposition \ref{p:main} and thus also the proof of Theorem \ref{thm:main} and, consequently, the proof of Theorem \ref{thm:nonuniq}.

\bibliographystyle{acm}
\bibliography{transport}

\end{document}